\documentclass[11pt]{amsart}
\usepackage[numbers]{natbib}

\usepackage{hyperref}

\usepackage{graphicx}
\usepackage{amsmath,amssymb,mathrsfs}
\usepackage{amsthm}

\newtheorem{thm}{Theorem}[section]
\newtheorem{cor}[thm]{Corollary}
\newtheorem{lem}[thm]{Lemma}
\newtheorem{prop}[thm]{Proposition}

\theoremstyle{definition}
\newtheorem{defn}[thm]{Definition}
\theoremstyle{remark}
\newtheorem{rem}[thm]{Remark}
\theoremstyle{example}
\newtheorem{exa}[thm]{Example}
\theoremstyle{conjecture}

\numberwithin{equation}{section}

\newcommand{\sig}{\sigma}

\newcommand{\im}{{\rm Im}\,}
\newcommand{\re}{{\rm Re}\,}

\newcommand {\X} {\mathbb X}
\newcommand {\R} {\mathbb R}
\newcommand {\C} {\mathbb C}

\newcommand {\A} {\mathbb A}
\newcommand {\Q} {\mathbb Q}

\newcommand {\calD} {\mathcal D}
\newcommand {\calF} {\mathcal F}

\newcommand {\calC} {\mathcal C}

\newcommand {\calH} {\mathcal H}

\newcommand {\mfd} {\mathbf D}

\newcommand {\bfe} {\mathbf e}

\begin{document}

\title[Complex symmetric evolution equations]{Complex symmetric evolution equations}%

\date{\today}%

\author{Pham Viet Hai}%
\address[P. V. Hai]{Faculty of Mathematics, Mechanics and Informatics, Vietnam National University, 334 Nguyen Trai, Thanh Xuan, Hanoi, Vietnam.}%
\email{phamviethai86@gmail.com}

\author{Mihai Putinar}
\address[M.~Putinar]{University of California at Santa Barbara, CA,
USA and Newcastle University, Newcastle upon Tyne, UK} 
\email{\tt mputinar@math.ucsb.edu, mihai.putinar@ncl.ac.uk}

\subjclass[2010]{47 D06, 47 B32, 30 D15}%

\keywords{complex symmetry, $C_0$-(semi)groups, Stone's theorem, evolution families, Fock space}%


\maketitle

\begin{abstract}
We study certain dynamical systems which leave invariant an indefinite quadratic form via semigroups or evolution families of complex symmetric Hilbert space operators. In the setting of bounded operators we show that a $\calC$-selfadjoint operator generates a contraction $C_0$-semigroup if and only if it is dissipative. In addition, we examine the abstract Cauchy problem for nonautonomous linear differential equations
possessing a complex symmetry. In the unbounded operator framework we isolate the class of \emph{complex symmetric, unbounded semigroups} and investigate Stone-type theorems adapted to them. On Fock space realization, we characterize all $\calC$-selfadjoint, unbounded weighted composition semigroups. As a byproduct we prove that the generator of a $\calC$-selfadjoint, unbounded semigroup is not necessarily $\calC$-selfadjoint.
\end{abstract}

\section{Introduction}
\subsection{Complex symmetric operators}
The study of abstract complex symmetric operators is relatively new, originating in the articles \cite{GP1, GP2}, although specific classes of these operators and their spectral behavior were investigated for several decades. Notable in this respect is the non-hermitian quantum mechanics formalism; its hamiltonians are complex symmetric, non-selfadjoint in the classical sense,  but have real spectrum. See \cite{GPP} for a survey of both abstract and concrete features of the analysis of complex symmetric operators.

We start by recalling some basic terminology, illustrated by a couple of examples.
\begin{defn}
An unbounded, linear operator $T$ is called \emph{$\calC$-symmetric} on a separable, complex Hilbert space $\calH$ if there exists a conjugation $\calC$ (i.e. an anti-linear, isometric involution) such that
\begin{equation}\label{C-symmetry}
\langle \calC x,Ty\rangle=\langle \calC Tx,y\rangle,\quad\forall x,y\in\text{dom}(T).
\end{equation}
\noindent Note that for a densely defined operator $T$, its adjoint $T^*$ is well-defined, and so the identity \eqref{C-symmetry} means that $T\preceq \calC T^*\calC$. A densely defined, linear operator $S$ is called \emph{$\calC$-selfadjoint} if it satisfies $S=\calC S^*\calC$.
\end{defn}

Alternatively, and closer to the quantum physics formalism, a $\calC$-symmetric operator $T$ leaves invariant the complex bilinear form $[x,y] = \langle x, C y \rangle$:
$$ [Tx, y] = [x, Ty], \ \ x,y \in dom(T).$$

Normal bounded operators, Toeplitz and Hankel finite matrices, Jordan forms, are all complex symmetric, with respect to an adapted and highly relevant conjugation. Moreover, every unitary transform of a Hilbert space is the product of two conjugations, and this observation enters deeply into the spectral structure of any $\calC$-selfadjoint operator \cite{GP1}.

As non-trivial examples we mention the  $\mathcal{PT}$-symmetric operators, a class advocated by Bender and Boettcher \cite{bender1998real} for its relevance to quantum mechanics. Roughly speaking, $\mathcal{PT}$-symmetric operators are those operators on Lebesgue space $L^2(\mathbf{R})$ complex symmetric with respect to the conjugation
$$
\mathcal{PT}f(x)=\overline{f(-x)},\quad f\in L^2(\mathbf{R}).
$$
The conjugation $\mathcal{PT}$ is the product of the \emph{time-reversal operator} $\mathcal{T}$, acting as $\mathcal{T}f(x)=\overline{f(x)}$ and the \emph{parity operator} $\mathcal{P}$, acting as $\mathcal{P}f(x)=f(-x)$. Nowadays the analysis of quantum systems possessing a $\mathcal{PT}$-symmetry is blooming. A special issue dedicated to this subject was published in the \emph{Journal of Physics A: Mathematical and Theoretical} vol. 45, No. 44 (2012) see \cite{bender2012quantum} for guiding light. See also the survey \cite{Znojil-2015} and the full Wiley volume it opens as manifesto.

The bibliography devoted to time dependent and evolution of non-hermitian hamiltoninans is meagre compared to the studies of stationary $\mathcal{PT}$-systems. In general a similarity to a hermitian operator is explored, along formal manipulations, not always bounded or convergent \cite{Miao-Xu,Scolarici}. So far, perturbation theory methods in Hilbert or Krein space offer the most rigorous approach to time evolution studies in this area \cite{Faria}. A notable inverse problem related to time dependent $\mathcal{PT}$-symmetric quantum systems was recently published \cite{Tamilselvan}.

In a recent work \cite{hai2018complex} we outlined a natural link between $\mathcal{PT}$-symmetric operators on Lebesgue space and linear differential operators acting on Fock space. More precisely, we realized the canonical conjugation $\mathcal{PT}$ as  $\mathcal{J}f(z)=\overline{f(\overline{z})}$ acting on Fock space. Via this dictionary we proved that a linear differential operator is $\mathcal{PT}$-selfadjoint on Lebesgue space if and only if it is unitarily equivalent to a linear combination with complex coefficients of operators of the form
$$ z^j [\frac{\partial}{\partial z}]^p + (-1)^{j+p}  z^p [\frac{\partial}{\partial z}]^j.$$
The above sum is endowed with the maximal domain of definition. 

\subsection{Bounded $C_0$-semigroups}
The theory of semigroups of bounded, linear operators has its origin in Stone's theorem on groups of unitary operators (or simply: \emph{unitary groups}). This theorem was motivated by the time-dependent solution of Schr\"odinger's equation in quantum mechanics (see \cite{KS}). The study of semigroups, rather than groups, was carried out in 1936 by Hille with the inspiration coming from the semigroup properties of certain classical singular integrals. However it was not until 1948 that the applicability of semigroups was fully appreciated. At that time Yosida used the semigroup theory to investigate diffusion equations. The same framework turned out to be relevant for Cauchy's problem for the wave equation and for investigating evolution equations. In the 1970s and 80s, thanks to the efforts of many research groups, this theory has reached maturity, see the monograph \cite{EN}. Today semigroups of operators are essential components of toolboxes of applied mathematicians working on integro-differential equations, functional differential equations, stochastic analysis or control theory.

\begin{defn}\label{defse}
A family $(U(s))_{s \geq 0}$ of \emph{bounded} linear operators on a Banach space $\X$, is called a \emph{strongly continuous semigroup} (or simply: \emph{bounded $C_0$-semigroup}), if
\begin{enumerate}
\item $U(0)=I$, the identity operator on $\X$;
\item $U(t+s) = U(t)U(s)$, $\forall t, s\geq 0$;
\item $\lim\limits_{s \to 0^+}U(s) x =x$, $\forall x \in \X$.
\end{enumerate}
\end{defn}
We remark the reader that in the above definition the terminology `bounded' means that each operator $U(t)$ is bounded. Furthermore, if there exists a constant $M>0$ such that
$$
\|U(s)\|\leq M,\quad\forall s\geq 0,
$$
then $(U(s))_{s \geq 0}$ is called \emph{uniformly bounded}. In particular with $M=1$, it is called a \emph{contraction}.
\begin{defn}\label{defgense}
The \emph{generator} $A\colon \text{dom}(A)\subseteq \X\to \X$ of a bounded $C_0$-semigroup $(U(s))_{s \geq 0}$ is defined by
$$
Ax=\lim\limits_{s\to0^+}\dfrac{U(s)x-x}{s},
$$
with the domain of definition
$$
\text{dom}(A)=\left\{x\in \X: \lim\limits_{s\to0^+}\frac{U(s)x-x}{s}\ \text{exists}\right\}.
$$
\end{defn}
Likewise, for $t,s\in\R$, we can extend Definitions \ref{defse}-\ref{defgense} to $C_0$-groups and their generators. If $(U(t))_{t \geq 0}$ is a bounded $C_0$-semigroup on a Banach space $\X$, then the family $(U(t)^*)_{t \geq 0}$, where for each $t\geq 0$, $U(t)^*$ is the dual (henceforth named however {\it adjoint}) of the operator $U(t)$, is called the \emph{bounded adjoint semigroup} acting on the dual space $\X^*$. In general, the family $(U(t)^*)_{t \geq 0}$ is not necessarily a bounded $C_0$-semigroup, but on a Hilbert space it is again a bounded $C_0$-semigroup (see \cite{partington2004linear}). Although the theory of bounded $C_0$-semigroups continued to develop rapidly, what concerns bounded adjoint semigroups, so far as we know, there seems to be very few works. The general theory of bounded adjoint semigroups was first studied systematically by Phillips \cite{phillips1955adjoint}, whose results are presented in the book of Hille and Phillips \cite{hille1996functional}, and was developed a little later by de Leeuw \cite{de1960adjoint}. Before that, Feller \cite{feller1953semi} had already used adjoint semigroups to investigate partial differential equations.

Aiming at a Stone type theorem, the first author isolated in \cite{HK4} the class of \emph{complex symmetric, bounded semigroups} acting on Hilbert space, and studied these semigroups on Fock space
realization. For the purpose of the present work, it is to recall some technical definitions.
\begin{defn}\label{defn-bff-cso-semi}
A bounded $C_0$-(semi)group $(T(t))_{t\geq 0}$ is called 

(i) \emph{$\calC$-symmetric} if each operator $T(t)$ is $\calC$-symmetric in the sense of linear bounded operators,

(ii) \emph{complex symmetric} if there exists some conjugation independent of $t$, such that it is $\calC$-symmetric.
\end{defn}

What make the complex symmetric, bounded $C_0$-groups interesting is the fact that they are generalizations of unitary groups; namely for a given unitary group $(T(t))_{t\in\mathbf{R}}$, one always finds a conjugation $\calC$ (independent of the parameter $t$) such that $T(t)=\calC T(t)^*\calC$ for all $t\in\mathbf{R}$. The reader can refer to \cite[Proposition 2.6]{HK4} for a detailed proof. The validity of a Stone-type theorem for complex symmetric, bounded $C_0$-(semi)groups is therefore in order. Indeed, if a bounded $C_0$-semigroup is $\calC$-symmetric, then its generator is $\calC$-selfadjoint in the sense of unbounded operators. For a proof see \cite[Theorem 2.4]{HK4}.

The present paper addresses a similar question for complex symmetric \emph{unbounded} $C_0$-semigroups with the expected conclusion that the unbounded operators framework is much more delicate to delineate.

\subsection{Evolution families of bounded linear operators}
The theory of operator semigroups emerged from the study of autonomous linear differential equations on infinite-dimensional Banach spaces. If an unbounded, linear operator $A$ generates a $C_0$-semigroup $(S(t))_{t\geq 0}$ (i.e. $A$ satisfies conditions in the Hille-Yosida theorem), then for each $x\in\text{dom}(A)$ the autonomous equation
$$
u'(t)=Au(t),\quad u(0)=x
$$
has a unique solution and furthermore this solution is given by $u(t)=S(t)x$.

Nonautonomous linear differential equations of the form
\begin{equation}\label{ACPt}
v'(t)=B(t)v(t),\quad v(s)=x\in\text{dom}[B(s)],\quad t\geq s\geq 0,
\end{equation}
where the domain $\text{dom}[B(s)]$ is assumed to be dense in a Hilbert space $\calH$, are also natural and widely referred to in applications.  

In this context, the following recent example investigated in \cite{BB-Phy} is relevant for our inquiry.

\begin{exa}[\cite{BB-Phy}]
Consider the time dependent non-hermitian Hamiltonian
$$
H(t)=\nu I+i\kappa(t)\sigma_z+\dfrac{\lambda(t)}{2}(\sigma_++\sigma_-),
$$
where $\nu$ is a real constant, and $\kappa,\lambda$ are real and continuous functions of $t$,
$$
\sigma_z=
  \left( {\begin{array}{cc}
   1 & 0 \\
   0 & -1 \\
  \end{array} } \right),\quad
\sigma_+=
  \left( {\begin{array}{cc}
   0 & 2 \\
   0 & 0 \\
  \end{array} } \right),\quad
\sigma_-=
  \left( {\begin{array}{cc}
   0 & 0 \\
   2 & 0 \\
  \end{array} } \right).
$$
Motivated by a quantum mechanical problem, Bagchi \cite{BB-Phy} established some conditions for the solutions of the equation $iv'(t)=H(t)v(t)$ to be of the form
$$v(t)=e^{-ia(t)}e^{ib(t)\sigma_+}e^{ic(t)\sigma_-}e^{d(t)\sigma_z}x.$$
Here $a,b,c$ are unknown real functions, to be determined. 
\end{exa}

Two observations stand out related to this example: 

1) The operators $H(t)$ are not hermitian, but they possess a $\mathcal{PT}$-symmetry.

2) The operator $e^{-ia(t)}e^{ib(t)\sigma_+}e^{ic(t)\sigma_-}e^{d(t)\sigma_z}$ is again $\mathcal{PT}$-symmetric. 

These features suggest a Stone type theorem holds when additional complex symmetry is present. It should be noted that the original Stone theorem and its recent extension \cite{HK4} to complex symmetric operators were so far proved only for (semi)groups; in other words, they refer to autonomous linear differential equations. This prompts us to examine below extensions of Stone's theorem to the nonautonomous case.

The leverage of the principal notions appearing in this article (complex symmetry, semigroups of operators, Cauchy problem) is clarified by their Fock space realization. A second half of our work is devoted to this interpretation.


\section{Outline and contents}
When dealing with bounded linear operators we show in Theorem \ref{bsdsbsbaf} that a $\calC$-selfadjoint operator generates a contraction $C_0$-semigroup if and only if it is dissipative. Theorem \ref{Ston-thm-nonauto} is an extension of Stone's theorem to the case of nonautonomous linear differential equations.

In the unbounded setting, we introduce the class of :\emph{complex symmetric, unbounded semigroups} and establish the following main results:

(i) A Stone-type theorem asserting that a $\calC$-selfadjoint unbounded $C_0$-semigroup possesses a $\calC$-symmetric generator (Theorems \ref{stone-type-thm1}-\ref{stone-1}). 

(ii)  In Fock space we characterize all $\calC$-selfadjoint, unbounded weighted composition semigroups. We derive the negative conclusion that the generator of a $\calC$-selfadjoint, unbounded semigroup is not necessarily $\calC$-selfadjoint. This example also implies that  if the operator $A$ generates an unbounded $C_0$-semigroup, then not always the adjoint $A^\ast$ generates the adjoint semigroup.

The rest of the paper is organized as follows. Section \ref{csbs} is devoted to studying complex symmetric, \emph{bounded} semigroups. Section \ref{csef} contains a brief detour through the nonautonomous linear differential equations enhanced by a complex symmetry.

To start the study in the unbounded case, we record in Section \ref{pre-unbdd-se} some preliminaries on unbounded $C_0$-semigroups. Stone-type theorems for complex symmetric, \emph{unbounded} semigroups are proved in Section \ref{sec4}. Our next task is to investigate complex symmetric, \emph{unbounded} semigroups on the Fock space of entire functions. To that aim, we recall in Section \ref{Fock-sp} some technical aspects related to Fock space. In Section \ref{last-sec}, we consider a family $(W_{\xi,\zeta}(t))_{t\geq 0}$ (often called as \emph{weighted composition semigroup}) of unbounded operators, where each $W_{\xi,\zeta}(t)$ is an unbounded weighted composition operator induced by the semiflow $(\zeta_t)_{t\geq 0}$ and the corresponding semicocycle $(\xi_t)_{t\geq 0}$; namely,
\begin{equation}\label{semif-wco}
(W_{\xi,\zeta}(t)f)(z):=\xi_t(z)f(\zeta_t(z)),\ z\in\mathbf{C}.
\end{equation}
The aim of this section is to characterize the family $(W_{\xi,\zeta}(t))_{t\geq 0}$ when it is a $\calC$-selfadjoint, unbounded semigroup on Fock space with respect to some weighted composition conjugation. The computation of generators of these semigroups is carried out in detail.

\section*{Notations}
We let $\mathbf{N}$, $\mathbf{R}$, $\mathbf{R}_+$ and $\mathbf{C}$ denote the sets of non-negative integers, real numbers, non-negative real numbers and complex numbers, respectively. The domain of an unbounded operator is denoted as $\text{dom}(\cdot)$ or $\text{dom}[\cdot]$. For two unbounded operators $X,Y$, the notation $X\preceq Y$ means that $X$ is the \emph{restriction} of $Y$ on the domain $\text{dom}(X)$; namely
$$
\text{dom}(X)\subseteq\text{dom}(Y),\quad Xz=Yz,\,\forall z\in\text{dom}(X).
$$
The product $XY$ is defined by
$$
\text{dom}(XY)=\{z\in\text{dom}(Y):Yz\in\text{dom}(X)\},\quad XYz=X(Yz),\,\forall z\in\text{dom}(XY).
$$
For a family $(F(t))_{t\in I}$ of unbounded, linear operators, we use the following symbols
$$D(F)=\bigcap_{t\in I}\text{dom}[F(t)],\quad\calD(F)=\bigcap_{t,s\in I}\text{dom}[F(t)F(s)].$$

\section{Complex symmetric, bounded semigroups}\label{csbs}
Before touching the unbounded case, we discuss some properties of the complex symmetric, bounded semigroups. We establish first a connection between these semigroups and the class of dissipative operators. This quest is motivated by the following reason: the generator of a contraction $C_0$-semigroup is maximal dissipative. Conversely, every operator with this property generates a contraction $C_0$-semigroup (see \cite[Chapter II: Theorem 3.5]{EN}). For the convenience of the reader, we recall some terminology.
\begin{defn}
A linear operator $A:\text{dom}(A)\subseteq\calH\to\calH$ is called 
\begin{enumerate}
\item \emph{maximal dissipative} if $(0,\infty)\subset\rho(A)$ and
\begin{equation}\label{cond-c-self-fhd}
\|\alpha R(\alpha,A)\|\leq 1,\quad\forall\alpha>0.
\end{equation}
\item \emph{dissipative} if
\begin{equation}\label{cond-c-self}
\|(\alpha I-A)z\|\geq\alpha\|z\|,\quad\forall\alpha>0,\forall z\in\text{dom}(A).
\end{equation}
\end{enumerate}
\end{defn}

\begin{rem}
It is well-known that every maximal dissipative operator is densely defined and closed (see \cite[Propositions 3.1.6 \& 3.1.11]{tucsnak2009observation}).
\end{rem}

\begin{prop}\label{kgfjd}
Let $A$ be a $\calC$-selfadjoint operator, and $\calC$ be a conjugation on a separable, complex Hilbert space. Then the operator $A$ is maximal dissipative if and only if it is dissipative.
\end{prop}
\begin{proof}
It is enough to show the implication $\Longleftarrow$. Indeed, we suppose that it holds that \eqref{cond-c-self}. Fix $\alpha>0$. It is easy to check that the operator $\alpha I-A$ is injective, and the image $\im (\alpha I-A)$ is closed. We show by contradiction that $\im (\alpha I-A)=\calH$. Assume that there exists some non-zero $x\in\calH\setminus\im (\alpha I-A)$ such that
$$
\langle x,(\alpha I-A)z\rangle=0,\quad\forall z\in\text{dom}(A).
$$
Equivalently, for every $z\in\text{dom}(A)$ we have
$$
\langle Az,x\rangle=\langle z,\alpha x\rangle,
$$
which implies that $x\in\text{dom}(A^*)$ and $A^*x=\alpha x$. Since the operator $A$ is $\calC$-selfadjoint, we have $\calC x\in\text{dom}(A)$ and $A\calC x=\alpha\calC x$. Again by \eqref{cond-c-self}, we conclude that $\calC x=0$, or equivalently $x=0$; but this is impossible. Thus, we infer that the operator $\alpha I-A$ is invertible with $\|R(\alpha,A)\|\leq \alpha^{-1}$.
\end{proof}

With the help of Proposition \ref{kgfjd}, we can give a characterization for a $\calC$-selfadjoint operator when it generates a contraction $C_0$-semigroup.
\begin{thm}\label{bsdsbsbaf}
Let $A$ be a $\calC$-selfadjoint operator, where $\calC$ is a conjugation on a separable, complex Hilbert space. Then the following assertions are equivalent.
\begin{enumerate}
\item The operator $A$ generates a contraction $C_0$-semigroup $(T(t))_{t\geq 0}$.
\item The operator $A$ is dissipative.
\end{enumerate}
Furthermore, $(T(t))_{t\geq 0}$ is $\calC$-symmetric.
\end{thm}
\begin{proof}
The equivalence follows from Proposition \ref{kgfjd}, while the complex symmetry of $(T(t))_{t\geq 0}$ follows from \cite[Theorem 2.2]{HK4}.
\end{proof}

Note that an isometric $C_0$-group is unitary and so by \cite[Proposition 2.6]{HK4} we have the following.
\begin{prop}\label{iso-group-cso}
Every isometric $C_0$-group is complex symmetric.
\end{prop}

We end the paper with a sufficient condition for isometric $C_0$-semigroups to be complex symmetric.
\begin{cor}
Let $(T(t))_{t\geq 0}$ be an isometric $C_0$-semigroup generated by the operator $A$. If the inclusion $\sigma(A)\subset i\mathbf{R}$ holds, then $(T(t))_{t\geq 0}$ is complex symmetric.
\end{cor}
\begin{proof}
By \cite[Chapter IV: Lemma 2.19]{EN}, $(T(t))_{t\geq 0}$ can be extended to an isometric $C_0$-group. By Proposition \ref{iso-group-cso}, we obtain the desired conclusion.
\end{proof}


\section{Complex symmetric evolution families}\label{csef}

Turning the page to non-autonomous linear differential equations, we first recall some basic terminology well explained in the monograph \cite{EN}.

\begin{defn}\label{defn-evo-fa}
A family $\{U(t,s)\}_{t\geq s\geq 0}$ of bounded, linear operators on a Hilbert space $\calH$ is called an \emph{evolution family} if
\begin{enumerate}
\item $U(t,t)=I$, $\forall t\geq 0$;
\item $U(t,s)=U(t,r)U(r,s)$, $\forall t\geq r \geq s$;
\item for each $x\in\Bbb\calH$, the mapping $(t,s)\mapsto U(t,s)x$ is continuous.
\end{enumerate}
\end{defn}

It is clear that a bounded $C_0$-semigroup $(S(t))_{t\geq 0}$ gives rise to the corresponding evolution family by setting $U(t,s)=S(t-s)$. In other terms, an evolution family $\{V(t,s)\}_{t\geq s\geq 0}$ satisfying $V(t-s,0)=V(t,s)$ arises from the bounded $C_0$-semigroup $Q(t)=V(t,0)$.

\begin{defn}
Let $s\geq 0$ and $x\in\text{dom}[B(s)]$. Then the function $v:[s,\infty)\to\calH$ is called a \emph{classical solution} of the abstract Cauchy problem \eqref{ACPt} if it is continuously differentiable such that $v(t)\in\text{dom}[B(t)]$ for every $t\geq s$ and it satisfies \eqref{ACPt}.
\end{defn}

Although both arising from linear differential equations, properties of evolution families are quite different from properties of semigroups. We recall a few differences. As mentioned, if $(S(t))_{t\geq 0}$ is a bounded $C_0$-semigroup on a Hilbert space, then its adjoint semigroup is again a bounded $C_0$-semigroup. In contrast, the family $\{U(t,s)^*\}_{t\geq s\geq 0}$ of adjoint operators is not an evolution family. The reader may wish to prove this claim by checking the axiom (2) of Definition \ref{defn-evo-fa}. The next difference lies on the existence question. Hille-Yosida theorem offers a characterization of the generator of a $C_0$-semigroup. To our knowledge, this is an open problem for evolution families. In fact, the existence problems was settled only in some special cases. In this section we do not touch the existence question. Instead, we assume that an evolution family is given and then try un unveil its complex symmetry.

Start with an evolution family \emph{solving} for the nonautonomous linear differential equation \eqref{ACPt}. It should be noted that there are various notions of solvability. In this framework, we follow Engel and Nagel \cite{EN}.
\begin{defn}[{\cite[page 479]{EN}}]
The evolution family of operators $\{U(t,s)\}_{t\geq s\geq 0}$ \emph{solves} the abstract Cauchy problem \eqref{ACPt} if there are dense subspaces $\calH_s, s\geq 0$ of $\calH$ such that
\begin{enumerate}
\item for every $t\geq s\geq 0$, the inclusions $U(t,s)\calH_s\subset\calH_t\subset\text{dom}[B(t)]$ hold;
\item for every $s\geq 0$ and every $z\in\calH_s$, the function $U(.,s)z$ is a classical solution of \eqref{ACPt}.
\end{enumerate}
In this case, the abstract Cauchy problem \eqref{ACPt} is called \emph{well-posed}.
\end{defn}
More restrictive definition can be found in \cite{FHO}, where Fattorini requires that $\calH_t=\text{dom}[B(t)]=M$, i.e. $\calH_t$ is independent of $t$.

Although the family $\{U(t,s)^*\}_{t\geq s\geq 0}$ of adjoint operators is not an evolution family, some properties are almost similar to evolution families. It turns out that under some suitable conditions, the function $U(.,s)^*z$ is a solution of the equation $w'(t)=B(t)^*w(t)$ with the initial condition $w(s)=z$. We make precisely this in the proposition below. Recall that for adjoint operators we denote $D(B^*)=\bigcap_{t\geq 0}\text{dom}[B(t)^*]$.
\begin{prop}\label{prop-U*}
Assume that $\{U(t,s)\}_{t\geq s\geq 0}$ solves the abstract Cauchy problem \eqref{ACPt}, and $U(\cdot,\cdot)^*$, $B(\cdot)$, $B(\cdot)^*$ are continuous in strong topology. Furthermore, assume that each operator $B(s)$ is densely defined. If $D(B^*)\ne\{0\}$, then the following properties hold.
\begin{enumerate}
\item For every $t,d\geq s\geq 0$ and every $z\in D(B^*)$, one has
$$
U(t,s)^*z-U(d,s)^*z=\int\limits_d^t U(\tau,s)^*B(\tau)^*z\,d\tau.
$$
\item For every $t\geq s\geq 0$ and every $z\in D(B^*)$, one has
$$
\dfrac{\partial}{\partial t}U(t,s)^*z=U(t,s)^*B(t)^*z,\quad\forall t\geq s.
$$
\item For every $s\geq 0$ and every $z\in D(B^*)$, one has
$$
\dfrac{\partial}{\partial t}U(t,s)^*z\bigg |_{t=s}=B(s)^*z,\quad\forall s\geq 0.
$$
\end{enumerate}
\end{prop}
\begin{proof}
It is clear that the parts (2-3) follows from the first part. We prove the first one as follows. Let $s\geq 0$, $x\in\calH_s$ and $z\in D(B^*)$. It follows from well-posedness that $U(.,s)x$ is a classical solution of the abstract Cauchy problem \eqref{ACPt}. Thus,
\begin{equation}\label{solution-ACPt}
U(t,s)x-U(d,s)x=\int\limits_d^t B(\tau)U(\tau,s)x\,d\tau,\quad\forall t,d\geq s.
\end{equation}
Since the family $\{U(t,s)\}_{t\geq s\geq 0}$ consists of bounded operators, we have
\begin{eqnarray*}
\langle x,U(t,s)^*z-U(d,s)^*z\rangle
&=&\langle U(t,s)x-U(d,s)x,z\rangle\\
&=&\int\limits_d^t \langle B(\tau)U(\tau,s)x,z\rangle\,d\tau\quad\text{(by \eqref{solution-ACPt})}\\
&=&\langle x,\int\limits_d^t U(\tau,s)^*B(\tau)^*z\,d\tau\rangle\quad\text{(as $z\in D(B^*)$)},
\end{eqnarray*}
which gives, as $\calH_s$ is dense, the desired conclusion.
\end{proof}

We are ready to examine an extension of Stone's theorem to the nonautonomous case.
\begin{thm}[Stone-type theorem]\label{Ston-thm-nonauto}
Assume that $\{U(t,s)\}_{t\geq s\geq 0}$ solves the abstract Cauchy problem \eqref{ACPt}, and $U(\cdot,\cdot)^*$, $B(\cdot)$, $B(\cdot)^*$ are continuous in strong topology. Let $\calC$ be a conjugation. Furthermore, assume that each operator $B(s)$ is densely defined. If $\{U(t,s)\}_{t\geq s\geq 0}$ is $\calC$-symmetric in the sense of bounded operators, then the following conclusions hold.
\begin{enumerate}
\item For every $s\geq 0, \ x\in\calH_s, \ y\in\calC(\calH_s)$, one has
$$
\langle B(s)x,y\rangle=\langle x,\calC B(s)\calC y\rangle;
$$
\item If $D(B^*)\ne\emptyset$, then for every $s\geq 0, f\in\calC(D(B^*)), g\in D(B^*)$, one has
$$
\langle f,B(s)^*g\rangle=\langle \calC B(s)^*\calC f,g\rangle.
$$
\end{enumerate}
\end{thm}
\begin{proof}
(1) Let $s\geq 0$, and $x,z\in\calH_s$. It follows from the well-posedness, that $U(.,s)x$ and $U(.,s)z$ are classical solutions of the abstract Cauchy problem \eqref{ACPt}. Thus,
$$
\dfrac{\partial}{\partial t}U(t,s)x=B(t)U(t,s)x,\quad\dfrac{\partial}{\partial t}U(t,s)z=B(t)U(t,s)z,\quad,\quad\forall t\geq s.
$$
Since $\{U(t,s)\}_{t\geq s\geq 0}$ is $\calC$-symmetric in the sense of bounded operators, for every $t\geq s\geq 0$ we can write
\begin{eqnarray*}
\langle \dfrac{U(t,s)x-U(s,s)x}{t-s},\calC z\rangle
&=&\langle x,\dfrac{U(t,s)^*-U(s,s)^*}{t-s}\calC z\rangle\\
&=&\langle x,\calC\left(\dfrac{U(t,s)z-U(s,s)z}{t-s}\right)\rangle.
\end{eqnarray*}
Letting $t\to s$ in the last equality gives
$$
\langle\dfrac{\partial}{\partial t}U(t,s)x\bigg |_{t=s},\calC z\rangle=\langle x,\calC\left(\dfrac{\partial}{\partial t}U(t,s)z\bigg |_{t=s}\right)\rangle.
$$
Consequently, taking into account that $U(s,s)=I$, we have
$$
\langle B(s)x,\calC z\rangle=\langle x,\calC B(s)z\rangle.
$$
(2) Let $s\geq 0$, and $h,g\in D(B^*)$. Also since $\{U(t,s)\}_{t\geq s\geq 0}$ is $\calC$-symmetric in the sense of bounded operators, for every $t\geq s\geq 0$ we can write
\begin{eqnarray*}
\langle \calC h,\dfrac{U(t,s)^* g-U(s,s)^*g}{t-s}\rangle
&=& \langle \dfrac{U(t,s)\calC h-U(s,s)\calC h}{t-s},g\rangle\\
&=&\langle \calC\left(\dfrac{U(t,s)^*-U(s,s)^*f}{t-s}\right),g\rangle.
\end{eqnarray*}
Letting $t\to s$ in the last equality gives
$$
\langle \calC h,\dfrac{\partial}{\partial t}U(t,s)^*g\bigg |_{t=s}\rangle=\langle \calC\left(\dfrac{\partial}{\partial t}U(t,s)^*h\bigg |_{t=s}\right),g\rangle,
$$
which implies, by Proposition \ref{prop-U*}, that
$$
\langle \calC h,B(s)^*g\rangle=\langle\calC B(s)^*h,g\rangle
$$
and the proof is complete.
\end{proof}

If $B(s)$ is bounded for all fixed $s\geq 0$, then there always exists an evolution family $\{U(t,s)\}_{t\geq s\geq 0}$ solving the abstract Cauchy problem \eqref{ACPt}. Furthermore, the Stone-type theorem in this case is simplified as follows.
\begin{cor}
Assume that $U(\cdot,\cdot)^*$, $B(\cdot)$, $B(\cdot)^*$ are continuous in strong topology. Furthermore, suppose that $B(s)$ is bounded for all fixed $s\geq 0$. Let $\calC$ be a conjugation. If $\{U(t,s)\}_{t\geq s\geq 0}$ is $\calC$-symmetric in the sense of bounded operators, then so is $(B(s))_{s\geq 0}$, i.e.
$$
B(s)=\calC B(s)^*\calC,\quad\forall s\geq 0.
$$
\end{cor}
\section{Preliminaries on unbounded $C_0$-semigroups}\label{pre-unbdd-se}
Technically speaking, the semigroups of unbounded, linear operators, proposed by Hughes \cite{RJH} satisfy the semigroup and strong continuity properties on a suitable subspace. Hughes generalized in \cite{RJH} the notion of a generator and proved a Hille-Yosida theorem to this unbounded setting. We recall below the basic definitions and some relevant results which will be referred to in the sequel.
\begin{defn}
A family $(T(t))_{t\geq 0}$ of unbounded, linear operators acting on a Banach space $\X$ is called an \emph{unbounded $C_0$-semigroup} if there exists $x$ such that 
\begin{enumerate}
\item[(A1)] $x\in\calD(T)\setminus\{0\}$;
\item[(A2)] $T(t)T(s)x=T(t+s)x$ for every $t,s\geq 0$;
\item[(A3)] $T(\cdot)x$ is continuous on $(0,\infty)$, and 
\begin{equation}\label{limit=0}
\lim\limits_{t\to 0^+}\|T(t)x-x\|=0.
\end{equation}
\end{enumerate}
Let $\mfd(T)$ denote the set of elements satisfying axioms (A1)-(A3).
\end{defn}

It is clear that $\mfd(T)\subseteq \calD(T)\subseteq D(T)$. From now on, for $\omega\in\mathbf{R}$ and $x\in D(T)$, we denote
\begin{equation}\label{N-Sigma}
N_\omega(x)=\sup\limits_{t\geq 0}e^{-\omega t}\|T(t)x\|,\quad \Sigma_\omega=\{x\in\mfd(T):N_\omega(x)<\infty\}.
\end{equation}
For $\omega\in\mathbf{R}$, $\lambda\in\mathbf{C}$ with $\re\lambda >\omega$ and $x\in\Sigma_\omega$, we define
\begin{equation}\label{J-omega}
J_\lambda^\omega x=\int_0^\infty e^{-\lambda t}T(t)xdt.
\end{equation}
It was proved in \cite[Theorem 2.9]{RJH} that there is a closed linear operator $A^\omega$ in $\Sigma_\omega$ (i.e. it is closed in the $N_\omega$-norm topology) for which the family $\{J_\lambda^\omega:\re\lambda >\omega\}$ on $\Sigma_\omega$ is the resolvent of $A^\omega$. In general, the operator $A^\omega$ is not closed in the norm of $\X$. It is elementary to check that if $\omega_1\leq\omega_2$, then $\Sigma_{\omega_1}\subseteq\Sigma_{\omega_2}$ and $A^{\omega_1}\preceq A^{\omega_2}$.

\begin{defn}
The \emph{generator} $A$ of an unbounded $C_0$-semigroup $(T(t))_{t\geq 0}$ is defined as follows: 
\begin{enumerate}
\item $\text{dom}(A)=\bigcup\limits_{\omega\in\R}\text{dom}(A^\omega)$;
\item if $x\in\text{dom}(A^\omega)$, and $x=J_\lambda^\omega y$ for $y\in\Sigma_\omega$, $\re\lambda>\omega$, then $Ax=\lambda x-y$.
\end{enumerate} 
\end{defn}

The proposition below gathers some key properties which are needed in later proofs.
\begin{prop}[{\cite[pages 121, 124]{RJH})}]\label{prop-2}
The following conclusions hold.
\begin{enumerate}
\item For each $\omega\in\mathbf{R}$, the domain $\text{dom}(A^\omega)$ is precisely
\begin{eqnarray}
\text{dom}(A^\omega)
\nonumber&&=\bigg\{x\in\Sigma_\omega:\,\text{$T(\cdot)x$ is differentiable at every $t>0$}\\
\label{Aomega-1-a}&&\quad\quad\text{and there exists $y\in\Sigma_\omega$ such that $y=\lim\limits_{t\to 0^+}\dfrac{T(t)x-x}{t}$}\bigg\}.
\end{eqnarray}
\item For $x\in\text{dom}(A)$, the limit $\lim\limits_{t\to 0^+}\dfrac{T(t)x-x}{t}$ exists in the norm of $\X$, and is equal to $Ax$ (see \cite[Theorem 2.13]{RJH}).
\item If $x\in\text{dom}(A)$, then $T(t)x\in\text{dom}(A)$ and the function $T(\cdot)x$ is differentiable with $dT(t)x/dt=T(t)Ax=AT(t)x$.
\end{enumerate}
\end{prop}

Let $\calD_\omega$ be the closure of $\text{dom}(A^\omega)$ in the $N_\omega$-norm; note that $\calD_\omega\subseteq\Sigma_\omega$. For each $\omega\in\mathbf{R}$, we define the operator $\A^\omega$ by
\begin{equation}\label{Aomega-2}
\text{dom}(\A^\omega)=\{x\in\calD_\omega:x\in\text{dom}(A^\omega),A^\omega x\in\calD_\omega\},\quad \A^\omega x=A^\omega x.
\end{equation}
This is an unbounded operator on $\calD_\omega$.

\begin{prop}[{\cite[Theorem 2.20, Lemma 2.25]{RJH}}]\label{Thm220}
The following statements are true:
\begin{enumerate}
\item $\calD_\omega$ is a Banach space with respect to the norm $N_\omega$.
\item For every $\omega\in\mathbf{R}$ and every $t>0$, we have $T(t)\calD_\omega\subseteq\calD_\omega$. Moreover, the family $(T(t)|_{\calD_\omega})_{t\geq 0}$ is a bounded $C_0$-semigroup acting on the Banach space $(\calD_\omega,N_\omega)$, with the generator $\A^\omega$.
\end{enumerate}
\end{prop}

\section{Stone-type theorems}\label{sec4}
In this section, we are concerned with an unbounded $C_0$-semigroup $(T(t))_{t\geq 0}$ acting on a complex Hilbert space $\calH$. Due to computations necessarily involving dual operators, we impose the assumption that $T(t)$ is densely defined for all fixed $t\geq 0$ (unless otherwise specified). As in the bounded case, we adopt the following definition.
\begin{defn}
If $(T(t))_{t \geq 0}$ is an unbounded $C_0$-semigroup on a complex Hilbert space $\calH$, then the family $(T(t)^*)_{t \geq 0}$, where for each $t\geq 0$, $T(t)^*$ is the adjoint of the operator $T(t)$, is called the \emph{unbounded adjoint semigroup}.
\end{defn}

\subsection{Some initial properties}
This section contains several technical observations which will be later on referred to. Some of these statements may have an intrinsic value.
\begin{thm}\label{thm-*}
Let $(T(t))_{t\geq 0}$ be an unbounded $C_0$-semigroup on a complex Hilbert space $\calH$. Assume that $T(t)$ is densely defined for all fixed $t\geq 0$. Then:
\begin{enumerate}
\item If the set $\calD(T)$ is dense, then the family $(T(t)^*)_{t\geq 0}$ satisfies axiom (A2) for every $x\in\calD(T^*)$;
\item $\lim\limits_{t\to s}\langle T(t)^*x-T(s)^*x,y\rangle=0$, $\lim\limits_{t\to 0^+}\langle T(t)^*x-x,y\rangle=0$, $\forall x\in D(T^*)$, $\forall y\in\mfd(T)$, $\forall s>0$.
\end{enumerate}
\end{thm}
\begin{proof}
(1) Let $x\in\calD(T^*)$.

Since the family $(T(t))_{t\geq 0}$ is an unbounded $C_0$-semigroup, by axiom (A2), for every $y\in\calD(T)$, we have $T(t)T(s)y=T(t+s)y$, and so,
$$
\langle T(t)T(s)y,x\rangle=\langle T(t+s)y,x\rangle.
$$
Note that
$$
\langle T(t)T(s)y,x\rangle=\langle y,T(s)^*T(t)^*x\rangle
$$
and
$$
\langle T(t+s)y,x\rangle=\langle y,T(t+s)^*x\rangle.
$$
Thus, we get
$$
\langle y,T(s)^*T(t)^*x\rangle=\langle y,T(t+s)^*x\rangle,\quad\forall y\in\calD(T),
$$
which implies, as the set $\calD(T)$ is dense, that $T(s)^*T(t)^*x=T(t+s)^*x$.

(2) We omit the case when $s>0$ and prove the case when $s=0$ as their techniques are similar. Let $x\in D(T^*)$. For every $y\in\mfd(T)$, we have
\begin{eqnarray*}
\langle T(t)^*x-x,y\rangle
&=&\langle x,T(t)y-y\rangle,
\end{eqnarray*}
which implies, as $(T(t))_{t\geq 0}$ is an unbounded $C_0$-semigroup, that
\begin{eqnarray}\label{claim-2}
\lim\limits_{t\to 0^+}\langle T(t)^*x-x,y\rangle=0,\quad\forall y\in\mfd(T).
\end{eqnarray}
\end{proof}

The concept of an adjoint pair of operators \cite[page 167]{TK} is useful for studying the generators of unbounded adjoint semigroups. Recall that two unbounded operators $S$ and $T$ are called \emph{adjoint to each other} if
$$
\langle Sx,y\rangle=\langle x,Ty\rangle,\quad\forall x\in\text{dom}(S),\forall y\in\text{dom}(T).
$$
\begin{thm}\label{B<A*}
Let $(T(t))_{t\geq 0}$ be an unbounded $C_0$-semigroup on a complex Hilbert space $\calH$ with generator $A$. Assume that $T(t)$ is densely defined for all fixed $t\geq 0$. If the family $(T(t)^*)_{t\geq 0}$ is an unbounded $C_0$-semigroup with generator $B$, then the two operators $A$ and $B$ are adjoint to each other.

Furthermore, if the operator $A$ is densely defined, then the operator inclusion $B\preceq A^*$ holds.
\end{thm}
\begin{proof}
(1) Let $x\in\text{dom}(A)\subseteq\mfd(T)$ and $y\in\text{dom}(B)\subseteq\mfd(T^*)$. We have
\begin{eqnarray*}
\langle \dfrac{T(t)x-x}{t},y\rangle=\langle x,\dfrac{T(t)^*y-y}{t}\rangle,
\end{eqnarray*}
which implies, by letting $t\to 0^+$ and using Proposition \ref{prop-2}(3), that
$$
\langle Ax,y\rangle=\langle x,By\rangle.
$$
In particular, if the operator $A$ is densely defined, then the adjoint $A^*$ is well-defined, and so the above identity gives $B\preceq A^*$.
\end{proof}

\begin{rem}
As mentioned in the Introduction, if $(T(t))_{t\geq 0}$ is a bounded $C_0$-semigroup with generator $A$, then the adjoint operator $A^*$ is always the generator of the bounded adjoint semigroup $(T(t)^*)_{t\geq 0}$. However, this fails to hold when the semigroup is unbounded (see Propositions \ref{generator-1}-\ref{generator-2}).
\end{rem}

\begin{prop}\label{prop-afternoon-sun}
Let $(T(t))_{t\geq 0}$ be an unbounded $C_0$-semigroup on a complex Hilbert space $\calH$ with generator $A$. Assume that $T(t)$ is densely defined for all fixed $t\geq 0$. Furthermore, assume that the generator $A$ is densely defined. Then:
\begin{enumerate}
\item For every $x\in D(T^*)$, 
$$\int_0^t T(s)^*x\,ds\in\text{dom}(A^*),\quad\forall t\geq 0$$
and
$$
A^*\left(\int_0^t T(s)^*x\,ds\right)=T(t)^*x-x,\quad\forall t\geq 0.
$$
\item If $(T(t)^*)_{t\geq 0}$ is a semigroup, then $\mathbf{D}(T^*)\subseteq\overline{\text{dom}(A^*)}$.
\end{enumerate}
\end{prop}
\begin{proof}
(1) Fix $t\geq 0$ and $x\in D(T^*)$.  For every $y\in\text{dom}(A)$ one finds
\begin{eqnarray*}
\langle \int_0^t T(s)^*x\,ds,Ay\rangle
&=&\int_0^t \langle T(s)^*x,Ay\rangle  ds=\int_0^t \langle x,T(s)Ay\rangle ds\\
&=&\langle x,\int_0^t T(s)Ay\,ds\rangle=\langle x,T(t)y-y\rangle\\
&=&\langle T(t)^*x-x,y\rangle,
\end{eqnarray*}
where the second and fourth equalities hold by Proposition \ref{prop-2}(3).

(2) Let $x\in \mathbf{D}(T^*)$. By the first part, we have
$$\dfrac{1}{t}\int_0^t T(s)^*x\,ds\in\text{dom}(A^*),$$
and so,
$$
x=\lim\limits_{t\to 0^+}T(s)^*x=\lim\limits_{t\to 0^+}\dfrac{1}{t}\int_0^t T(s)^*x\,ds\in\overline{\text{dom}(A^*)}.
$$
\end{proof}

The results below are simple, but they will be important steps toward studying the Stone-type theorems for unbounded $C_0$-semigroups.
\begin{prop}\label{T<V}
Let $(T(t))_{t\geq 0}$, $(V(t))_{t\geq 0}$ be two families of unbounded, linear operators on a separable, complex Hilbert space $\calH$ such that $T(t)\preceq V(t)$ for every $t\geq 0$. If $(T(t))_{t\geq 0}$ is an unbounded $C_0$-semigroup, then so is $(V(t))_{t\geq 0}$. In this case, for every $\omega>0$, $A^\omega\preceq B^\omega$, where $A,B$ are generators of $(T(t))_{t\geq 0}$, $(V(t))_{t\geq 0}$, respectively.
\end{prop}
\begin{proof}
Since $T(t)\preceq V(t)$ for every $t\geq 0$, we can check that $\mfd(T)\subseteq\mfd(V)$, and hence if the set $\mfd(T)$ is non-empty, then so is $\mfd(V)$.

Again since $T(t)\preceq V(t)$ for every $t\geq 0$, we have
$$
\Sigma_\omega(T)\subseteq\Sigma_\omega(V),\quad N_{\omega,T}(x)=N_{\omega,V}(x),\,\forall x\in\Sigma_\omega(T).
$$
By \cite[Theorem 2.15]{RJH}, the generator $A$ is of the following form
$$\text{dom}(A)=\underset{\omega\in\mathbf{R}}{\bigcup}\text{dom}(A^\omega),\quad Af=A^\omega f,\, f\in\text{dom}(A^\omega),$$
where the operators $A^\omega$, $\omega\in\mathbf{R}$, are defined by \eqref{Aomega-1-a}, that is $A^\omega f=g$,
\begin{eqnarray*}
\text{dom}(A^\omega)&=&\bigg\{f\in\Sigma_\omega(T):\,\text{$T(\cdot)f$ is differentiable at every $t>0$}\\
\nonumber&&\quad\text{and there exists $g\in\Sigma_\omega(T)$ such that $g=\lim\limits_{t\to 0^+}\dfrac{T(t)f-f}{t}$}\bigg\}.
\end{eqnarray*}
It follows from the fact $\Sigma_\omega(T)\subseteq\Sigma_\omega(V)$ and $T(t)\preceq V(t)$, that
\begin{eqnarray*}
\text{dom}(A^\omega)
&\subseteq&\bigg\{f\in\Sigma_\omega(V):\,\text{$T(\cdot)f$ is differentiable at every $t>0$}\\
\nonumber&&\quad\text{and there exists $g\in\Sigma_\omega(V)$ such that $g=\lim\limits_{t\to 0^+}\dfrac{T(t)f-f}{t}$}\bigg\}\\
&\subseteq&\text{dom}(B^\omega),
\end{eqnarray*}
and furthermore,
$$
A^\omega f=B^\omega f,\quad\forall f\in\text{dom}(A^\omega).
$$
\end{proof}

\begin{prop}\label{jnbv}
Let $P$ be a (linear or anti-linear) isometric involution acting on a complex Hilbert space $\calH$. If the family $(T(t))_{t\geq 0}$ is an unbounded $C_0$-semigroup with generator $A$, then the family $(S(t))_{t\geq 0}$ defined by $S(t)=PT(t)P$ is also an unbounded $C_0$-semigroup with generator $PAP$.
\end{prop}
\begin{proof}
Since the operator $P$ is an isometric involution, we have $\mfd(S)=P[\mfd(T)]$, and hence the family $(S(t))_{t\geq 0}$ is also an unbounded $C_0$-semigroup on $\calH$.

Suppose that the operator $B$ is the generator of $(S(t))_{t\geq 0}$. By \cite[Theorem 2.15]{RJH}, the generator $B$ is of the following form
$$\text{dom}(B)=\underset{\omega\in\mathbf{R}}{\bigcup}\text{dom}(B^\omega),\quad Bf=B^\omega f,\, f\in\text{dom}(B^\omega),$$
where the operators $B^\omega$, $\omega\in\mathbf{R}$, are defined by \eqref{Aomega-1-a}, that is $B^\omega f=g$,
\begin{eqnarray*}
\text{dom}(B^\omega)&=&\bigg\{f\in\Sigma_\omega(S):\,\text{$S(\cdot)f$ is differentiable at every $t>0$}\\
\nonumber&&\quad\text{and there exists $g\in\Sigma_\omega(S)$ such that $g=\lim\limits_{t\to 0^+}\dfrac{S(t)f-f}{t}$}\bigg\}.
\end{eqnarray*}
Also since the operator $P$ is involutive, for every $\omega\in\mathbf{R}$ we have $\Sigma_\omega(T)=P[\Sigma_\omega(S)]$.

$\bullet$ First we prove that $B\preceq PAP$.

Let $f\in\text{dom}(B)$. Then there exists $\omega\in\mathbf{R}$ such that $f\in\text{dom}(B^\omega)$. Fix $s>0$. Since $S(\cdot)f$ is differentiable at $s$ with $x=dS(s)f/ds$, for every $\varepsilon>0$, there exists some $\delta=\delta(\varepsilon,s)$ such that
$$
\left\|\dfrac{S(t)f-S(s)f}{t-s}-x\right\|<\varepsilon,\quad\forall |t-s|<\delta.
$$
Thus, for every $t$ with $|t-s|<\delta$, we have
\begin{eqnarray*}
\left\|\dfrac{T(t)Pf-T(s)Pf}{t-s}-Px\right\|
&=&\left\|\dfrac{PT(t)Pf-PT(s)Pf}{t-s}-x\right\|\\
&=&\left\|\dfrac{S(t)f-S(s)f}{t-s}-x\right\|<\varepsilon.
\end{eqnarray*}
The above inequality means that the function $T(\cdot)Pf$ is differentiable at $s$.

We have
$$
Bf=B^\omega f=\lim\limits_{t\to 0^+}\dfrac{S(t)f-f}{t}=\lim\limits_{t\to 0^+}\dfrac{PT(t)Pf-f}{t}\in\Sigma_\omega(S),
$$
and hence
$$
PBf=P\left[\lim\limits_{t\to 0^+}\dfrac{PT(t)Pf-f}{t}\right]=\lim\limits_{t\to 0^+}\dfrac{T(t)Pf-Pf}{t}.
$$
Note that $Pf\in P[\Sigma_\omega(S)]=\Sigma_\omega(T)$ and $PBf\in P[\Sigma_\omega(S)]=\Sigma_\omega(T)$. 

Thus, by \cite[Theorem 2.15]{RJH}, $Pf\in\text{dom}(A^\omega)$, and furthermore, $APf=PBf$, as wanted.

$\bullet$ Next, we prove that $PAP\preceq B$.

Let $f\in\text{dom}(PAP)$, which means that $Pf\in\text{dom}(A)$. Then there exists $\omega\in\mathbf{R}$ such that $Pf\in\text{dom}(A^\omega)$. Fix $s>0$. Then the function $T(\cdot)Pf$ is differentiable at $s$ with $y=dT(s)Pf/dt$. By the definition, for every $\varepsilon>0$, there exists $\delta=\delta(s,\varepsilon)>0$ such that
$$
\left\|\dfrac{T(t)Pf-T(s)Pf}{t-s}-y\right\|<\varepsilon,\quad\forall |t-s|<\delta,
$$
Thus, for every $t$ with $|t-s|<\delta$, we have
\begin{eqnarray*}
\left\|\dfrac{S(t)f-S(s)f}{t-s}-Py\right\|
&=&\left\|\dfrac{PT(t)Pf-PT(s)Pf}{t-s}-Py\right\|\\
&=&\left\|\dfrac{T(t)Pf-T(s)Pf}{t-s}-y\right\|<\varepsilon.
\end{eqnarray*}
The above inequality means that the function $S(\cdot)f$ is differentiable at $s$.

We infer
$$
APf=A^\omega Pf=\lim\limits_{t\to 0^+}\dfrac{T(t)Pf-Pf}{t}\in\Sigma_\omega(T),
$$
and hence
\begin{eqnarray*}
PAPf &=& P\left[\lim\limits_{t\to 0^+}\dfrac{T(t)Pf-Pf}{t}\right]=\lim\limits_{t\to 0^+}\dfrac{PT(t)Pf-f}{t}\\
&=&\lim\limits_{t\to 0^+}\dfrac{S(t)f-f}{t}\in P[\Sigma_\omega(T)]=\Sigma_\omega(S),
\end{eqnarray*}
where the second equality holds as the operator $P$ is bounded. Note that that $Pf\in\text{dom}(A^\omega)\subseteq\Sigma_\omega(T)$ implies $f\in P[\Sigma_\omega(T)]=\Sigma_\omega(S)$.

Thus, by \cite[Theorem 2.15]{RJH}, $f\in\text{dom}(B^\omega)$, and furthermore $Bf=PAPf$, as claimed.
\end{proof}

\subsection{Main results}
With all preparation in place, we turn to a Stone-type theorem for complex symmetric, unbounded semigroups.
\begin{defn}\label{defn-bff-cso-semi-unbdd}
Let $\calC$ be a conjugation on a separable, complex Hilbert space $\calH$. An unbounded $C_0$-(semi)group $(T(t))_{t\geq 0}$ is called (i) \emph{$\calC$-symmetric} if each operator $T(t)$ is $\calC$-symmetric; (ii) \emph{$\calC$-selfadjoint} if each operator $T(t)$ is $\calC$-selfadjoint.
\end{defn}

In view of the definition of complex symmetric operators, we separate the discussion into three cases:

(i) $T(t)$ is non-densely defined and is $\calC$-symmetric.

(ii) $T(t)$ is densely defined and is $\calC$-symmetric.

(ii) $T(t)$ is densely defined and is $\calC$-selfadjoint.

- In the first case, to discuss the complex symmetry we must use the identity \eqref{C-symmetry}.
\begin{thm}\label{stone-type-thm1}
Let $(T(t))_{t\geq 0}$ be an unbounded $C_0$-semigroup on a separable, complex Hilbert space $\calH$ with generator $A$, and $\calC$ a conjugation. If $(T(t))_{t\geq 0}$ is $\calC$-symmetric, then the generator $A$ is $\calC$-symmetric.
\end{thm}
\begin{proof}
Since $(T(t))_{t\geq 0}$ is $\calC$-symmetric, for every $x,y\in\text{dom}(A)\subseteq\text{dom}[T(t)]$ we have
$$
\langle \calC x,T(t)y\rangle=\langle \calC T(t)x,y\rangle,
$$
which gives
$$
\langle \calC x,\dfrac{T(t)y-y}{t}\rangle=\langle \dfrac{\calC [T(t)x-x]}{t},y\rangle.
$$
Letting $t\to 0^+$, we get
$$
\langle \calC x,Ay\rangle=\langle \calC Ax,y\rangle.
$$
Thus, the generator $A$ is $\calC$-symmetric.
\end{proof}

- For the second situation, the adjoint $T(t)^*$ is well-defined and hence the $\calC$-symmetry of $(T(t))_{t\geq 0}$ is equivalent to the fact that $T(t)\preceq\calC T(t)^*\calC$ for every $t\geq 0$. Thus, we can use this inclusion to explore more properties regarding to the generator of a complex symmetric, unbounded semigroup.
\begin{thm}\label{prop-stone-thm}
Let $(T(t))_{t\geq 0}$ be an unbounded $C_0$-semigroup on a separable, complex Hilbert space $\calH$ with generator $A$, and $\calC$ a conjugation. Assume that $T(t)$ is densely defined for all fixed $t\geq 0$. If $(T(t))_{t\geq 0}$ is $\calC$-symmetric, then
\begin{enumerate}
\item $(T(t)^*)_{t\geq 0}$ is an unbounded $C_0$-semigroup;
\item $A\preceq \calC B\calC$, where $B$ is the generator of $(T(t)^*)_{t\geq 0}$;
\item the operator $A$ is $\calC$-symmetric.
\end{enumerate} 
\end{thm}

\begin{proof}
Let us define the family $(V(t))_{t\geq 0}$ be setting $V(t)=\calC T(t)^*\calC$.

(1) Since $T(t)\preceq \calC T(t)^*\calC=V(t)$ , by Proposition \ref{T<V}, the family $(V(t))_{t\geq 0}$ is an unbounded $C_0$-semigroup, and hence by Proposition \ref{jnbv} the family $(T(t)^*)_{t\geq 0}$ is also an unbounded $C_0$-semigroup.

(2) Proposition \ref{jnbv} shows the generator of $(V(t))_{t\geq 0}$ is precisely $\calC B\calC$. Thus, this conclusion follows directly from Proposition \ref{T<V}.

(3) This conclusion follows from Theorem \ref{stone-type-thm1}.
\end{proof}

- For the last case, we derive the equality of the operator inclusion stated in Theorem \ref{prop-stone-thm}(2).
\begin{thm}\label{stone-1}
Let $(T(t))_{t\geq 0}$ be an unbounded $C_0$-semigroup on a separable, complex Hilbert space with generator $A$, and $\calC$ a conjugation. Assume that $T(t)$ is densely defined for all fixed $t\geq 0$. If the family $(T(t))_{t\geq 0}$ is $\calC$-selfadjoint, then 
\begin{enumerate}
\item $(T(t)^*)_{t\geq 0}$ is an unbounded $C_0$-semigroup;
\item $A= \calC B\calC$, where $B$ is the generator of $(T(t)^*)_{t\geq 0}$;
\item the operator $A$ is $\calC$-symmetric.
\end{enumerate} 
\end{thm}
\begin{proof}
The conclusions (1) and (3) follow directly from Theorem \ref{prop-stone-thm}, while the conclusion (2) holds by Proposition \ref{jnbv}.
\end{proof}

\begin{rem}
We end this section with the note that the generator of a $\calC$-selfadjoint, unbounded semigroup is not necessarily $\calC$-selfadjoint. The next section, devoted to complex symmetry in Fock space, will provide examples supporting this statement.
\end{rem}

\section{Preliminaries on Fock space}\label{Fock-sp}
\subsection{Fock space}
The \emph{Fock space} $\calF^2$ (sometimes called the \emph{Segal-Bargmann space}) consists of entire functions which are square integrable with respect to the Gaussian measure $\frac{1}{\pi}e^{-|z|^2}\;dV(z)$, where $dV$ is the Lebesgue measure on $\mathbf{C}$. This is a reproducing kernel Hilbert space, with inner product
$$
\langle f,g\rangle=\dfrac{1}{\pi}\int_{\mathbf{C}} f(z)\overline{g(z)}e^{-|z|^2}\;dV(z),
$$
and kernel function
$$
K_z(u)=e^{u\overline{z}},\quad z,u\in\mathbf{C}.
$$
The set $\{\bfe_k\}_{k\in\mathbf{N}}$, where $\bfe_k(z)=z^k$, is an orthogonal basis for $\calF^2$ with
\begin{equation*}
\langle \bfe_m,\bfe_n\rangle=
\begin{cases}
0,\quad\text{if $m\ne n$},\\
n!,\quad\text{if $m=n$}.
\end{cases}
\end{equation*}
It was proved in \cite{HK1}, that the Fock space $\calF^2$ carries a three-parameter family of anti-linear, isometric involutions. These conjugations, known as \emph{weighted composition conjugations}, are described as follows
\begin{equation}\label{Ca,b,c-wcc}
\calC_{a,b,c}f(z)=ce^{bz}\overline{f\left(\overline{az+b}\right)},\quad f\in\calF^2,
\end{equation}
where $a,b,c$ are complex constants satisfying
\begin{equation}\label{abc-cond}
|a|=1, \quad \bar{a}b+\bar{b}=0, \quad |c|^2e^{|b|^2}=1,
\end{equation}

\subsection{Weighted composition operator}
Consider formal \emph{weighted composition expressions} of the form
$$
E(\psi,\varphi)f=\psi\cdot f\circ\varphi,
$$
where $\psi,\varphi$ are entire functions. Operator theorists are interested in the operators arising from the formal expression $E(\psi,\varphi)$ in $\calF^2$. One of such operators is the \emph{maximal weighted composition operator} defined by
$$
\text{dom}(W_{\psi,\varphi,\max})=\{f\in\calF^2:E(\psi,\varphi)f\in\calF^2\},
$$
$$
W_{\psi,\varphi,\max}f=E(\psi,\varphi)f,\quad\forall f\in\text{dom}(W_{\psi,\varphi,\max}).
$$
The domain $\text{dom}(W_{\psi,\varphi,\max})$ is called \emph{maximal}. The operator $W_{\psi,\varphi,\max}$ is ``maximal" in the sense that it cannot be extended as an operator in $\calF^2$ generated by the expression $E(\psi,\varphi)$ (see \cite{hai2018some}). The specification of the domain is crucial when dealing with unbounded linear operators. Considered on different domains, the same formal expression may generate operators with completely different properties. This observation prompts us to consider the weighted composition expressions on subspaces of the maximal domain. The operator $W_{\psi,\varphi}$ is called an \emph{unbounded weighted composition operator} if $W_{\psi,\varphi}\preceq W_{\psi,\varphi,\max}$; namely the domain $\text{dom}(W_{\psi,\varphi})$ is a subspace of the maximal domain $\text{dom}(W_{\psi,\varphi,\max})$, and the operator $W_{\psi,\varphi}$ is the restriction of the maximal operator $W_{\psi,\varphi,\max}$ on $\text{dom}(W_{\psi,\varphi})$.

A characterization for bounded weighted composition operators was given in \cite{TL}, where duality techniques play a key role. In \cite{HK2}, the authors provided a different proof, which does not refer to adjoint operators. For later use, we recall a particular form from \cite{HK2}.
\begin{prop}[\cite{HK2}]\label{exaa}
Let $\varphi(z)=Az+B$, $\psi(z)=Ce^{Dz}$, where $A,B,C,D$ are complex constants. The operator $W_{\psi,\varphi}$ is bounded on $\calF^2$ if and only if
    \begin{enumerate}
    \item either $|A|<1$,
    \item or $|A|=1$, $D+A\overline{B}=0$.
    \end{enumerate}
\end{prop}

A characterization for an unbounded weighted composition operator which is $\calC$-selfadjoint with respect to the weighted composition conjugation $\calC_{a,b,c}$ (or simply: \emph{$\calC_{a,b,c}$-selfadjoint}) was carried out in \cite{hai2018some}. It turns out that a $\calC_{a,b,c}$-selfadjoint, unbounded weighted composition operator must be necessarily maximal and the symbols $\psi,\varphi$ can be precisely computed.
\begin{prop}[{\cite{hai2018some}}]\label{abc-selfadjointness}
Let $W_{\psi,\varphi}$ be an unbounded weighted composition operator, induced by the symbols $\psi$, $\varphi$ with $\psi\not\equiv 0$. Furthermore, let $\calC_{a,b,c}$ be a weighted composition conjugation. Then the operator $W_{\psi,\varphi}$ is $\calC_{a,b,c}$-selfadjoint if and only if the following conditions hold.
\begin{enumerate}
\item $W_{\psi,\varphi}=W_{\psi,\varphi,\max}$.
\item The symbols are of the following forms 
\begin{equation*}
\varphi(z)=Az+B,\quad\psi(z) = Ce^{Dz},\ \hbox{with $C\ne 0,\ D=aB-bA+b$}.
\end{equation*}
\end{enumerate}
\end{prop}

\subsection{Weighted composition semigroups}
To construct $\calC_{a,b,c}$-selfadjoint, unbounded semigroups on the Fock space $\calF^2$, we rely on semigroups of weighted composition operators (or simply: \emph{weighted composition semigroups}). The notions of semiflows and semicocycles are important in defining these semigroups.

\begin{defn}\label{semiflow}
A family $(\zeta_t)_{t \geq 0}$ of nonconstant entire functions on $\mathbf{C}$ is called a \emph{semiflow} if
\begin{enumerate}
\item $\zeta_0(z)=z$, $\forall z \in \mathbf{C}$;
\item $\zeta_{t+s}(z)=\zeta_t ( \zeta_s (z))$, $\forall t, s \geq 0,$ $\forall z \in\mathbf{C}$.
\end{enumerate}
Likewise, if $t$, $s \in\mathbf{R}$, then it is called a \emph{flow}.
\end{defn}

A trivial example of a semiflow is $\zeta_t(z)=z,\ \forall t\ge0$. For a nontrivial semiflow, its structure was given by the following proposition.

\begin{prop}[\cite{JTTY}]\label{class}
If the family $(\zeta_t)_{t \geq 0}$ is a nontrivial semiflow on $\mathbf{C}$, then it satisfies
\begin{equation}
\text{either\quad}\zeta_t(z) = z+Et,
\end{equation}
\begin{equation}
\text{or\quad}\zeta_t(z) = e^{t\ell}z+G(1-e^{t\ell}),
\end{equation}
where $E,G,\ell$ are complex constants with $E\ne 0$ and $\ell\ne 0$. 
\end{prop}

\begin{defn}
Let $(\zeta_t)_{t \geq 0}$ be a semiflow. A family $(\xi_t)_{t \geq 0}$ of entire functions on $\mathbf{C}$ is called a \emph{semicocycle} for $(\zeta_t)_{t \geq 0}$ if
\begin{enumerate}
\item the mapping $t \mapsto \xi_t(z)$ is differentiable for every $z \in\mathbf{C}$;
\item $\xi_{t+s}=\xi_t \cdot (\xi_s \circ \zeta_t)$, $\forall t, s \geq 0$;
\item $\xi_0=1$.
\end{enumerate}
\end{defn}

Let $(\zeta_t)_{t \geq 0}$ be a semiflow, and $(\xi_t)_{t \geq 0}$ be the corresponding semicocycle. Recall that a \emph{weighted composition semigroup} is defined by \eqref{semif-wco}, i.e.
$$
(W_{\xi,\zeta}(t)f)(z):=\xi_t(z)f(\zeta_t(z)),\ z\in\mathbf{C}.
$$

We recall a result from \cite{HK4}, in which the authors succeeded to characterize the family $(W_{\xi,\zeta}(t))_{t\ge0}$ when it is a $\calC_{a,b,c}$-symmetric, bounded semigroup on $\calF^2$.
\begin{prop}[{\cite[Theorem 4.5]{HK4}}]\label{semi-wco-fock}
Let $(W_{\xi,\zeta}(t))_{t\ge0}$ be a family defined by \eqref{semif-wco}. Then it is a $\calC_{a,b,c}$-symmetric, bounded semigroup on $\calF^2$ if and only if 
\begin{equation}\label{T(t)-ABCD-bdd}
\zeta_t(z)=A(t)z+B(t),\quad \xi_t(z)=C(t)e^{zD(t)},
\end{equation}
where the functions $A,B,C,D:\R_+\to\mathbf{C}$ satisfy
\begin{equation}\label{form1-bdd}
\text{either\,\,}
\begin{cases}
A(t)=1,\quad B(t)=Et,\quad D(t)=aEt,\\
C(t)=\exp\left(Ft+aE^2t^2/2\right),
\end{cases}
\end{equation}
\begin{equation}\label{form2-bdd}
\text{or\,\,}
\begin{cases}
A(t)=e^{\ell t},\quad B(t)=G(1-e^{\ell t}),\quad D(t)=(aG+b)(1-e^{\ell t}),\\
C(t)=\exp\left[ Ht+G(aG+b)(e^{\ell t}-\ell t-1)\right].
\end{cases}
\end{equation}
Here, $\ell,E,F,G,H$ are complex constants satisfying the following conditions
\begin{enumerate}
\item[(B1)] $\ell\ne 0$, $E\ne 0$;
\item[(B2)] $\overline{E}+aE=0$;
\item[(B3)] either $\re\ell<0$, or $\re\ell=0$, $aG+b-\overline{G}=0$.
\end{enumerate}
\end{prop}

Note that in view of Proposition \ref{exaa}, the conditions (B1-B3) in Proposition \ref{semi-wco-fock} are exactly the characterization for the weighted composition operator $W_{\xi,\zeta}(t)$ to be bounded on $\calF^2$. With these conditions, $W_{\xi,\zeta}(t)$ is $\calC_{a,b,c}$-symmetric in the sense of bounded operators if and only if it is $\calC_{a,b,c}$-symmetric on polynomials. Thus, conditions (B1-B3) play an indispensable role in proving Proposition \ref{semi-wco-fock}.

\section{Complex symmetric, unbounded semigroup on Fock space}\label{last-sec}
\subsection{Some initial properties}
The following lemma was proved in \cite[Lemma 4.1]{HK4} by using the function $\log(\cdot)$. We propose below a proof, which avoids logarithmic functions.

\begin{lem}\label{firstlem}
Suppose that the differentiable function $\Lambda:\mathbf{R}_+\to\mathbf{C}$ satisfy
$$
\begin{cases}
\Lambda(t+s)=\Lambda(t) \Lambda(s) e^{\Psi(t,s)},\ \forall t,s\ge0,\\
\Lambda(0)=1,\ \Psi(t,0)=0,
\end{cases}
$$
where $\Psi:\mathbf{R}_+\times\mathbf{R}_+\to\mathbf{C}$ is a differentiable function of two variables. Then the function $\Lambda(\cdot)$ must be of the form
$$
\Lambda(t)=\exp\left(t\Lambda '(0)+\int_0^t \frac{\partial\Psi}{\partial s}(\tau,0)d\tau\right).
$$
\end{lem}

\begin{proof}
	We can rewrite
	$$
	\dfrac{\Lambda(t+s)-\Lambda(t)}{s}=\Lambda(t)\left(\dfrac{\Lambda(s)-1}{s}\right)e^{\Psi(t,s)}+\Lambda(t)\left(\dfrac{e^{\Psi(t,s)}-e^{\Psi(t,0)}}{s}\right).
	$$
	Letting $s\to 0$ gives
	\begin{eqnarray}\label{E'-correction}
	\Lambda '(t) 
	      &=& \Lambda(t)\left(\Lambda '(0)+\dfrac{\partial\Psi}{\partial s}(t,0)\right)\quad\text{(since $\Psi(t,0)=0$).}
	\end{eqnarray}
	For setting
	$$
    \Gamma(t)=\Lambda(t)\exp\left(-t\Lambda '(0)-\int_0^t \frac{\partial\Psi}{\partial s}(\tau,0)d\tau\right),
	$$
	we have $\Gamma(0)=\Lambda(0)=1$, and furthermore
	\begin{eqnarray*}
    \Lambda '(t)
         &=& \Gamma '(t)\exp\left(t\Lambda '(0)+\int_0^t \frac{\partial \Psi}{\partial s}(\tau,0)d\tau\right)+\Lambda(t)\left(\Lambda '(0)+\dfrac{\partial \Psi}{\partial s}(t,0)\right).
	\end{eqnarray*}
	Hence, equation \eqref{E'-correction} is exactly
	$$
	\Gamma '(t)\exp\left(t\Lambda '(0)+\int_0^t \frac{\partial F}{\partial s}(\tau,0)d\tau\right)=0\Longleftrightarrow \Gamma(t)=\Gamma(0)=1.
	$$
	We substitute $\Gamma(\cdot)$ back into $\Psi(\cdot)$ to get the desired form.
\end{proof}

The following proposition provides a necessary condition for the family $(W_{\xi,\zeta}(t))_{t\geq 0}$ when each operator $W_{\xi,\zeta}(t)$ is $\calC_{a,b,c}$-selfadjoint in the sense of unbounded operators. Its proof makes use of Proposition \ref{abc-selfadjointness}.
\begin{prop}\label{---->}
If the operators $W_{\xi,\zeta}(t)$, $t\geq 0$ are $\calC_{a,b,c}$-selfadjoint on $\calF^2$, then 
\begin{enumerate}
\item for every $t\geq 0$, $W_{\xi,\zeta}(t)=W_{\xi,\zeta}(t)_{\max}$;
\item the symbols are of the following forms
\begin{equation}\label{T(t)-ABCD}
\zeta_t(z)=A(t)z+B(t),\quad \xi_t(z)=C(t)e^{zD(t)},
\end{equation}
where the functions $A,B,C,D:\mathbf{R}_+\to\mathbf{C}$ satisfy
\begin{equation}\label{form1}
\text{either\,\,}
\begin{cases}
A(t)=1,\quad B(t)=Et,\quad D(t)=aEt,\\
C(t)=\exp\left(Ft+aE^2t^2/2\right),
\end{cases}
\end{equation}
\begin{equation}\label{form2}
\text{or\,\,}
\begin{cases}
A(t)=e^{\ell t},\quad B(t)=G(1-e^{\ell t}),\quad D(t)=(aG+b)(1-e^{\ell t}),\\
C(t)=\exp\left[ Ht+G(aG+b)(e^{\ell t}-\ell t-1)\right].
\end{cases}
\end{equation}
Here, $\ell,E,F,G,H$ are complex constants with $\ell, E\ne 0$.
\end{enumerate}
\end{prop}
\begin{proof}
Since the operators $W_{\xi,\zeta}(t)$, $t\geq 0$ are $\calC_{a,b,c}$-selfadjoint on $\calF^2$, by Proposition \ref{abc-selfadjointness}, $D(t)=aB(t)-bA(t)+b$. It follows from $\xi_{t+s}=\xi_t \cdot (\xi_s \circ \zeta_t)$, that
$$
C(t+s)=C(t)C(s)e^{B(t)D(s)},\quad\forall t,s\geq 0.
$$
In view of Proposition \ref{class}, there are two cases of $(\zeta_t)_{t\geq 0}$.

{\bf Case 1:} If $\xi_t(z)=z+Et$, then $A(t)=1$, $B(t)=Et$, and hence
$$D(t)=aB(t)-bA(t)+b=aEt.$$
Thus,
$$
C(t+s)=C(t)C(s)e^{aE^2 ts},\quad\forall t,s\geq 0,
$$
and hence by Lemma \ref{firstlem}, we get the explicit form of the function $C(\cdot)$.

{\bf Case 2:} If $\xi_t(z)=e^{t\ell}z+G(1-e^{t\ell})$, then $A(t)=e^{t\ell}$, $B(t)=G(1-e^{t\ell})$, and so $D(t)=(aG+b)(1-e^{\ell t})$.
Thus,
$$
C(t+s)=C(t)C(s)e^{G(1-e^{t\ell})(aG+b)(1-e^{\ell s})},\quad\forall t,s\geq 0,
$$
and hence by Lemma \ref{firstlem}, we get the explicit form of the function $C(\cdot)$.
\end{proof}

It turns out that the $\calC_{a,b,c}$-symmetry is also a sufficient condition for the family $(W_{\xi,\zeta}(t))_{t\geq 0}$ to be an unbounded $C_0$-semigroup. To prove this, we show that the monomials $\bfe_k$, where $k\in\mathbf{N}$, belong to the set $\mfd(W_{\xi,\zeta})$. Namely, we take turns checking the axioms (A1)-(A3). It is clear that $\bfe_k$ always satisfies the axiom (A1). The verification of the remaining axioms requires a closer look.

- For the axiom (A2), we have the following.
\begin{prop}\label{<--1}
Let $(\xi_t)_{t\geq 0}$, $(\zeta_t)_{t\geq 0}$ be families of entire functions given by either \eqref{form1} or \eqref{form2}. Then for every $k\geq 1$ and every $t,s\geq 0$, we have $\bfe_k\in\text{dom}(W_{\xi,\zeta}(s)_{\max}W_{\xi,\zeta}(t)_{\max})$, and moreover,
$$
W_{\xi,\zeta}(s)_{\max}W_{\xi,\zeta}(t)_{\max}\bfe_k=W_{\xi,\zeta}(s+t)_{\max}\bfe_k.
$$
\end{prop}
\begin{proof}
For every $k\geq 1$, we have $E(\xi_t,\zeta_t)\bfe_k(z)=C(t)e^{zD(t)}[A(t)z+B(t)]^k$, which gives
\begin{eqnarray*}
&& E(\xi_s,\zeta_s)E(\xi_t,\zeta_t)\bfe_k(z)\\
&&=\xi_s(z)E(\xi_t,\zeta_t)\bfe_k(\zeta_s(z))\\
&&= C(s)e^{zD(s)}C(t)e^{[A(s)z+B(s)]D(t)}[A(t)(A(s)z+B(s))+B(t)]^k\\
&&= C(s)C(t)e^{B(s)D(t)}e^{z[D(s)+A(s)D(t)]}[A(t)A(s)z+A(t)B(s)+B(t)]^k\\
&&= E(\xi_{s+t},\zeta_{s+t})\bfe_k(z).
\end{eqnarray*}
\end{proof}

- The verification of axiom (A3) is more involved than checking (A2).
\begin{prop}\label{<--2}
Let $(\zeta_t)_{t\geq 0}$ be a semiflow on $\mathbf{C}$, such that its corresponding semicocycle has the form $\xi_t(z) = C(t)e^{D(t)z}$, with  conditions
\begin{equation}\label{atbt}
\lim_{t\to s}C(t)=C(s), \quad\lim_{t\to s}D(t)=D(s),\quad\forall s>0,
\end{equation}
and
\begin{equation}\label{atbt-2}
\lim_{t\to 0^+}C(t)=C(0), \quad\lim_{t\to 0^+}D(t)=D(0).
\end{equation}
Then for every $m\in\mathbf{N}$, the monomial $\bfe_m$ satisfies axiom (A3).
\end{prop}

\begin{proof}
We omit the case when $s=0$ and prove the case when $s>0$, as the first case is rather simple. Fix $s>0$. For each $\bfe_m$, we can write
\begin{eqnarray*}
&&W_{\xi,\zeta}(t)_{\max}\bfe_m(z)-W_{\xi,\zeta}(s)_{\max}\bfe_m(z)\\
&&= [C(t)-C(s)]e^{zD(s)}\zeta_s(z)^m + C(t)[e^{zD(t)}-e^{zD(s)}]\zeta_s(z)^m\\
&&\quad + \,\xi_t(z)[\zeta_t(z)^m-\zeta_s(z)^m].
\end{eqnarray*}

- Note that the first term always tends to $0$ in $\calF^2$ as $t\to s$.

We consider the functions
$$
J_{t,s}(z)=C(t)[e^{zD(t)}-e^{zD(s)}]\zeta_s(z)^m,\quad K_{t,s}(z)=\xi_t(z)[\zeta_t(z)^m-\zeta_s(z)^m].
$$
Let $M(s)=\max\{|A(s)|,|B(s)|,|C(s)|+1, |D(s)|+1\}$. By \eqref{atbt}, we can find $\delta=\delta(s)\in\min\{1,s\}$ such that
\begin{equation}\label{bt2at1}
\sup\limits_{s-\delta \leq t \leq s+\delta}|C(t)|\leq M(s),\quad \sup\limits_{s-\delta \leq t \leq s+\delta}|D(t)|\leq M(s).
\end{equation}
For the rest of the proof, we let $t\in [s-\delta,s+\delta]$.

- Clearly, $J_{t,s}(\cdot)$ is an entire function. Note that
$$
|\zeta_s(z)|\leq M(s)(1+|z|),
$$
and for every $k\geq 1$, we have
\begin{eqnarray*}
|D(t)^k-D(s)^k|
&=&|D(t)-D(s)|\cdot \left|\sum_{j=0}^{k-1} D(t)^j D(s)^{k-1-j}\right|\\
&\leq& |D(t)-D(s)|k M(s)^{k-1}.
\end{eqnarray*}
Since
\begin{eqnarray*}
J_{t,s}(z) &=& C(t)\sum_{k\ge1}\frac{[D(t)^k-D(s)^k]z^k\zeta_s(z)^m}{k!},
\end{eqnarray*}
we can estimate
\begin{eqnarray*}
|J_{t,s}(z) | &\leq& M(s)\sum_{k\geq 1}\dfrac{|D(t)-D(s)|k M(s)^{k-1} |z|^kM(s)^m(1+|z|)^m}{k!}\\
&=& |D(t)-D(s)|M(s)^{m+1}|z|(1+|z|)^m e^{M(s)|z|}\\
&\leq & |D(t)-D(s)|M(s)^{m+1}(1+|z|)^{m+1} e^{M(s)|z|},
\end{eqnarray*}
and so,
\begin{eqnarray*}
\|J_{t,s}\|^2 &\leq&|D(t)-D(s)|^2 M(s)^{2m+2}\int_{\mathbf{C}}(1+|z|)^{2m+2} e^{2M(s)|z|-|z|^2}\,dV(z).
\end{eqnarray*}
Since the right-hand-side integral is finite, the function $J_{t,s}(\cdot)\in\calF^2$, and moreover it converges to $0$ as $t\to s$.

- Finally, we prove that $\lim\limits_{t\to s}K_{t,s} = 0$ in $\calF^2$-norm.

This limit is trivial if $m=0$. So, we only consider the case when $m\geq 1$. For this case, we note that
\begin{eqnarray*}
\|K_{t,s}\|^2 
&=& \frac{1}{\pi}\int_{\mathbf{C}}|C(t)|^2\left|\zeta_t(z)^m-\zeta_s(z)^m\right|^2 e^{2\re[zD(t)]-|z|^2}\;dV(z)\\
&\leq& \frac{M(s)^2}{\pi}\int_{\mathbf{C}} \left| \zeta_t(z)^m-\zeta_s(z)^m \right|^2 e^{2|zD(t)|-|z|^2}dV(z)\\
&\leq& \frac{M(s)^2}{\pi}\int_{\mathbf{C}} \left| \zeta_t(z)^m-\zeta_s(z)^m \right|^2 e^{2M(s)|z|-|z|^2}dV(z).
\end{eqnarray*}

In view of Proposition \ref{class}, there are two cases of the semiflow $(\zeta_t)$.

{\bf Case 1:} If $\zeta_t(z)=z+Et$ for some $E\in\mathbf{C}\setminus \{0\}$, then
$$
\zeta_t(z)^m-\zeta_s(z)^m=(z+Et)^m-(z+Es)^m=E(t-s)\sum_{j=0}^{m-1}(z+Et)^j(z+Es)^{m-1-j},
$$
and hence
$$
|\zeta_t(z)^m-\zeta_s(z)^m|\leq |E(t-s)|m[|z|+|E|(s+\delta)]^{m-1}.
$$
Thus, we can estimate
\begin{eqnarray*}
\|K_{t,s}\|^2 
&\leq& \frac{M(s)^2}{\pi}\int_{\mathbf{C}} \left| \zeta_t(z)^m-\zeta_s(z)^m \right|^2 e^{2M(s)|z|-|z|^2}dV(z)\\
&\leq& \frac{M(s)^2m^2|E(t-s)|^2}{\pi}\int_{\mathbf{C}}  [|z|+|E|(s+\delta)]^{2m-2}e^{2M(s)|z|-|z|^2}dV(z).
\end{eqnarray*}
The last inequality shows that $K_{t,s}(\cdot)\in\calF^2$, and moreover it converges to $0$ in $\calF^2$ as $t\to s$.

{\bf Case 2:} If $\zeta_t(z)=e^{t\ell}(z-G)+G$ for some $G\in\mathbf{C}$ and some $\ell \in\mathbf{C}\setminus\{0\}$, then
\begin{eqnarray*}
|\zeta_t(z)|&\leq& e^{t|\ell|}(|z|+|G|)+|G|\leq e^{(s+\delta)|\ell|}(|z|+|G|)+|G|\\
&\leq& L(s)[|z|+L(s)]\leq [|z|+L(s)]^2,
\end{eqnarray*}
where $L(s)=\max\{e^{(s+\delta)|\ell|},|G|+1\}$. Since
$$
\zeta_t(z)^m-\zeta_s(z)^m = (e^{t\ell}-e^{s\ell})(z-G)\sum_{j=0}^{m-1}\zeta_t(z)^j\zeta_s(z)^{m-1-j},
$$
we can estimate
\begin{eqnarray*}
|\zeta_t(z)^m-\zeta_s(z)^m|&\leq& |e^{t\ell}-e^{s\ell}|\cdot|z-G| m [|z|+L(s)]^{2m-2}\\
&\leq& |e^{t\ell}-e^{s\ell}|m [|z|+L(s)]^{2m-1},
\end{eqnarray*}
and hence
\begin{eqnarray*}
\|K_{t,s}\|^2 
&\leq& \frac{M(s)^2}{\pi}\int_{\mathbf{C}} \left| \zeta_t(z)^m-\zeta_s(z)^m \right|^2 e^{2M(s)|z|-|z|^2}dV(z)\\
&\leq& \frac{m^2M(s)^2|e^{t\ell}-e^{s\ell}|^2}{\pi}\int_{\mathbf{C}} [|z|+L(s)]^{4m-2} e^{2M(s)|z|-|z|^2}dV(z).
\end{eqnarray*}
The last inequality shows that $K_{t,s}\in\calF^2$, and moreover it converges to $0$ in $\calF^2$ as $t\to s$. The proof is complete.
\end{proof}

\subsection{Characterization}
With all preparation in place, we now state and prove the main result of this section.
\begin{thm}\label{main-res-unbdd}
Let $(W_{\xi,\zeta}(t))_{t\geq 0}$ be a weighted composition semigroup induced by the semiflow $(\zeta_t)_{t\geq 0}$ and the corresponding semicocycle $(\xi_t)_{t\geq 0}$. Then $(W_{\xi,\zeta}(t))_{t\geq 0}$ is a $\calC_{a,b,c}$-selfadjoint, unbounded semigroup if and only if 
\begin{enumerate}
\item for every $t\geq 0$, $W_{\xi,\zeta}(t)=W_{\xi,\zeta}(t)_{\max}$.
\item $(\zeta_t)_{t\geq 0}$, $(\xi_t)_{t\geq 0}$ are of the forms \eqref{T(t)-ABCD}, where the functions $A,B,C,D$ satisfy either forms \eqref{form1} or forms \eqref{form2}.
\end{enumerate}
\end{thm}
\begin{proof}
The necessity holds by Proposition \ref{---->}, while the sufficiency follows from Propositions \ref{<--1}-\ref{<--2}.
\end{proof}

Comparing Proposition \ref{semi-wco-fock} to Theorem \ref{main-res-unbdd}, there is an essential difference between the bounded and unbounded semigroup cases. The first case depends heavily on conditions (B1-B3).

\subsection{Generators}
In this subsection, we compute the generators of complex symmetric semigroups characterized in Theorem \ref{main-res-unbdd}. We recall some notations
$$
N_\omega(f)=\sup\limits_{t\geq 0}e^{-\omega t}\|W_{\xi,\zeta}(t)f\|,\quad \Sigma_\omega=\{f\in\mfd(W_{\xi,\zeta}):N_\omega(f)<\infty\}.
$$
In view of Theorem \ref{main-res-unbdd}, we consider two cases for the semiflow $(\zeta_t)_{t\geq 0}$.
\begin{prop}\label{generator-1}
Let $(W_{\xi,\zeta}(t))_{t\geq 0}$ be a $\calC_{a,b,c}$-selfadjoint, unbounded semigroup, which is of forms \eqref{T(t)-ABCD}-\eqref{form1}. Then 
\begin{enumerate}
\item the generator $Q$ of the semigroup $(W_{\xi,\zeta}(t))_{t\geq 0}$ is exactly
$$
\text{dom}(Q)=\underset{\omega\in\mathbf{R}}{\bigcup}\{f\in\Sigma_\omega:(F+aEz)f(z)+Ef'(z)\in\Sigma_\omega\},
$$
$$
Q f(z)=(F+aEz)f(z)+Ef'(z),\quad f\in\text{dom}(Q).
$$
\item The operator $Q$ is $\calC_{a,b,c}$-symmetric. 
\item If $\re(aE^2+|E|^2)>0$, then 
\begin{enumerate}
\item the operator $Q$ is not $\calC_{a,b,c}$-selfadjoint.
\item the operator $Q^*$ is not the generator of the adjoint semigroup $(W_{\xi,\zeta}(t)^*)_{t\geq 0}$.
\end{enumerate}
\item The point spectrum $\sigma_p(Q)=\emptyset$.
\end{enumerate}
\end{prop}
\begin{proof}
(1) Let us define the operator $X_1:\text{dom}(X_1)\subseteq\calF^2\to\calF^2$ by setting
$$
\text{dom}(X_1)=\underset{\omega\in\mathbf{R}}{\bigcup}\{f\in\Sigma_\omega:(F+aEz)f(z)+Ef'(z)\in\Sigma_\omega\},
$$
$$
X_1 f(z)=(F+aEz)f(z)+Ef'(z),\quad f\in\text{dom}(X_1).
$$
If $\omega_1\leq\omega_2$, then $\Sigma_{\omega_1}\subseteq\Sigma_{\omega_2}$, and hence,
\begin{eqnarray*}
&&\{f\in\Sigma_{\omega_1}:(F+aEz)f(z)+Ef'(z)\in\Sigma_{\omega_1}\}\\
&&\subseteq \{f\in\Sigma_{\omega_2}:(F+aEz)f(z)+Ef'(z)\in\Sigma_{\omega_2}\}.
\end{eqnarray*}
Thus, the operator $X_1$ is well-defined. For each $\omega\in\mathbf{R}$, we define the operator
$$
\text{dom}(X_{1,\omega})=\{f\in\Sigma_\omega:(F+aEz)f(z)+Ef'(z)\in\Sigma_\omega\},
$$
$$
X_{1,\omega}f=(F+aEz)f(z)+Ef'(z),\quad f\in\text{dom}(X_{1,\omega}).
$$

By \cite[Theorem 2.15]{RJH}, the generator $Q$ is of the following form
$$\text{dom}(Q)=\underset{\omega\in\mathbf{R}}{\bigcup}\text{dom}(Q^\omega),\quad Qf=Q^\omega f,\, f\in\text{dom}(Q^\omega),$$
where the operators $Q^\omega$, $\omega\in\mathbf{R}$ are defined as in \eqref{Aomega-1-a}, that is
\begin{eqnarray*}
\text{dom}(Q^\omega)&=&\bigg\{f\in\Sigma_\omega:\,\text{$W_{\xi,\zeta}(\cdot)f$ is differentiable at every $t>0$}\\
\nonumber&&\quad\text{and there exists $g\in\Sigma_\omega$ such that $g=\lim\limits_{t\to 0^+}\dfrac{W_{\xi,\zeta}(t)f-f}{t}$}\bigg\},
\end{eqnarray*}
$$
Q^\omega f=\lim\limits_{t\to 0^+}\dfrac{W_{\xi,\zeta}(t)f-f}{t}.
$$
Let $\calD_\omega$ be the closure of $\text{dom}(Q^\omega)$ in the $N_\omega$-norm. For each $\omega\in\mathbf{R}$, we define the operator $\Q^\omega$ as in \eqref{Aomega-2}, that is
$$
\text{dom}(\Q^\omega)=\{f\in\calD_\omega:f\in\text{dom}(Q^\omega),Q^\omega f\in\calD_\omega\},\quad\Q^\omega f=Q^\omega f,\,f\in\text{dom}(\Q^\omega).
$$
Let $f\in\text{dom}(Q)$. Then there exists $\omega\in\mathbf{R}$ such that $f\in\text{dom}(Q^\omega)$. We have
$$
Q f=Q^\omega f=\lim_{t\to 0^+}\dfrac{W_{\xi,\zeta}(t)f-f}{t}\in\Sigma_\omega,
$$
which gives
$$
Q f(z)=\lim_{t\to 0^+}\dfrac{W_{\xi,\zeta}(t)f(z)-f(z)}{t},\quad\forall z\in\mathbf{C}.
$$
Consequently, taking into account the expressions of $(\zeta_t)_{t\geq 0}$ and $(\xi_t)_{t\geq 0}$ in \eqref{form1}, we obtain
$$
Q f(z)=\dfrac{d}{dt}W_{\xi,\zeta}(t)f(z)\bigg\arrowvert_{t=0}=(F+aEz)f(z)+Ef'(z),\quad\forall z\in\mathbf{C},
$$
which implies $Q\preceq X_1$, and hence,
$$
\Q^\omega\preceq Q^\omega\preceq X_{1,\omega},\quad\forall\omega\in\mathbf{R}.
$$
As mentioned in Proposition \ref{Thm220}, the operator $\Q^\omega$ is the generator of the $C_0$-semigroup $(W_{\xi,\zeta}(t)|_{\calD_\omega})_{t\geq 0}$ of bounded, linear operators acting on the Banach space $(\calD_\omega,N_\omega)$. By \cite[Proposition I.5.5, Generation Theorem II.3.8]{EN}, there are constants $\alpha\geq 1$, $\theta\in\mathbf{R}$ such that
$$
N_\omega (W_{\xi,\zeta}(t)f)\leq \alpha e^{\theta t}N_\omega(f),\quad\forall t\geq 0,\forall f\in\calD_\omega,
$$
and moreover, $(\theta,\infty)\subseteq\rho(\Q^\omega)$.

Let $\lambda\in (\theta,\infty)$. Then the operator $\Q^\omega-\lambda I$ is onto. Let $f\in\ker(X_{1,\omega}-\lambda I)$. Then $f\in\text{dom}(X_{1,\omega})\subseteq\calF^2$, and
$$
(F+aEz-\lambda)f(z)+Ef'(z)=(X_{1,\omega}-\lambda I)f(z)=0,\quad\forall z\in\mathbf{C}.
$$
The equation above has the general solution
$$
f(z)=f(0)\exp\left(\dfrac{(\lambda-F)z}{E}-\dfrac{az^2}{2}\right).
$$
Since $|a|=1$, by \cite[Theorem 1.1]{KHI}, $f\in\calF^2$ if and only if $f(0)=0$, or equivalently $f=\mathbf{0}$. Thus, the operator $X_{1,\omega}-\lambda I$ is one-to-one. By \cite[Lemma 1.3]{KS}, we must have $X_{1,\omega}-\lambda I=\Q^\omega-\lambda I$, and hence, $\Q^\omega=Q^\omega= X_{1,\omega}$.

(2) This conclusion follows directly from Theorem \ref{stone-1}.

(3) Suppose that $\re[aE^2+|E|^2]>0$. 

(3-a) We prove by a contradiction that the operator $Q$ is not $\calC_{a,b,c}$-selfadjoint. Indeed, assume that the operator $Q$ is $\calC_{a,b,c}$-selfadjoint, and hence, by \cite{hai2018complex}, it must be maximal, that is
$$
\text{dom}(Q)=\{f\in\calF^2:(F+aEz)f(z)+Ef'(z)\in\calF^2\}.
$$
It is clear that $\mathbf{1}\in\{f\in\calF^2:(F+aEz)f(z)+Ef'(z)\in\calF^2\}=\text{dom}(Q)$, which means that there exists $\omega\in\mathbf{R}$ such that
$$
\mathbf{1}\in\{f\in\Sigma_\omega:(F+aEz)f(z)+Ef'(z)\in\Sigma_\omega\}\subseteq\Sigma_\omega.
$$
Since $W_{\xi,\zeta}(t)\mathbf{1}(z)=e^{Ft+aE^2 t^2/2+zaEt}=e^{Ft+aE^2 t^2/2}K_{t\overline{aE}}(z)$, we have
$$
\|W_{\xi,\zeta}(t)\mathbf{1}\|=e^{t\re F+t^2\re(aE^2)/2+t^2|aE|^2/2}=e^{t\re F+t^2\re(aE^2)/2+t^2|E|^2/2},
$$
and so
$$
\sup_{t\geq 0}e^{-\omega t}\|W_{\xi,\zeta}(t)\mathbf{1}\|=\sup_{t\geq 0}e^{(\re F-\omega)t+t^2\re(aE^2+|E|^2)/2}=\infty.
$$
Thus, $\mathbf{1}\notin\Sigma_\omega$, but this is impossible.

(3-b) Assume that the operator $R^*$ is the generator of the adjoint semigroup $(W_{\xi,\zeta}(t)^*)_{t\geq 0}$. By Theorem \ref{stone-1}(2), the operator $R$ is $\calC_{a,b,c}$-selfadjoint; but this is impossible by (3-a).

(4) Assume that $\sig_p(Q)\ne\emptyset$. Take $\eta \in \sig_p(Q)$. By definition of a point spectrum, there is $f \in \text{dom}(Q)\setminus \{0\}\subseteq\calF^2\setminus \{0\}$ such that $Q f=\eta f$, which gives
$$
Ef'(z)=(\eta-F-aEz)f(z).
$$
This differential equation has the general solution
$$
f(z)=f(0)e^{\frac{(\eta-F)z}{E}-\frac{az^2}{2}}\in\calF^2\setminus \{0\}.
$$
But this contradicts \cite[Theorem 1.1]{KHI}.
\end{proof}

\begin{prop}\label{generator-2}
Let $(W_{\xi,\zeta}(t))_{t\geq 0}$ be a $\calC_{a,b,c}$-selfadjoint, unbounded semigroup, which is of forms \eqref{T(t)-ABCD} and \eqref{form2}. Then 
\begin{enumerate}
\item The generator $R$ of the semigroup $(W_{\xi,\zeta}(t))_{t\geq 0}$ is exactly
$$
\text{dom}(R)=\underset{\omega\in\mathbf{R}}{\bigcup}\{f\in\Sigma_\omega:[H-\ell (aG+b)z]f+\ell (z-G)f'\in\Sigma_\omega\},
$$
$$
R f(z)=[H-\ell (aG+b)z]f(z)+\ell (z-G)f'(z),\quad z\in\mathbf{C}.
$$
\item The operator $R$ is $\calC_{a,b,c}$-symmetric.
\item If $\ell >0$ and $aG+b\ne 0$, then 
\begin{enumerate}
\item the operator $R$ is not $\calC_{a,b,c}$-selfadjoint.
\item the operator $R^*$ is not the generator of the adjoint semigroup $(W_{\xi,\zeta}(t)^*)_{t\geq 0}$.
\end{enumerate}
\item The point spectrum satisfies $\sigma_p(R)\subseteq\{H-\ell (aG+b)G+k\ell:k\in\mathbf{N}\}$.
\end{enumerate}
\end{prop}
\begin{proof}
(1) Let us define the operator $X_2$ by setting
$$
\text{dom}(X_2)=\underset{\omega\in\mathbf{R}}{\bigcup}\{f\in\Sigma_\omega:[H-\ell (aG+b)z]f+\ell (z-G)f'\in\Sigma_\omega\},
$$
$$
X_2 f(z)=[H-\ell (aG+b)z]f(z)+\ell (z-G)f'(z),\quad z\in\mathbf{C}.
$$
If $\omega_1\leq\omega_2$, then $\Sigma_{\omega_1}\subseteq\Sigma_{\omega_2}$, and hence,
\begin{eqnarray*}
&&\{f\in\Sigma_{\omega_1}:[H-\ell (aG+b)z]f+\ell (z-G)f'\in\Sigma_{\omega_1}\}\\
&&\subseteq \{f\in\Sigma_{\omega_2}:[H-\ell (aG+b)z]f+\ell (z-G)f'\in\Sigma_{\omega_2}\}.
\end{eqnarray*}
Thus, the operator $X_2$ is well-defined. 

By \cite[Theorem 2.15]{RJH}, the generator $R$ is of the following form
$$\text{dom}(R)=\underset{\omega\in\mathbf{R}}{\bigcup}\text{dom}(R^\omega),\quad Rf=R^\omega f,\, f\in\text{dom}(R^\omega),$$
where the operators $R^\omega$, $\omega\in\mathbf{R}$ are defined as in \eqref{Aomega-1-a}, that is
\begin{eqnarray*}
\text{dom}(R^\omega)&=&\bigg\{f\in\Sigma_\omega:\,\text{$W_{\xi,\zeta}(\cdot)f$ is differentiable at every $t>0$}\\
\nonumber&&\quad\text{and there exists $g\in\Sigma_\omega$ such that $g=\lim\limits_{t\to 0^+}\dfrac{W_{\xi,\zeta}(t)f-f}{t}$}\bigg\},
\end{eqnarray*}
$$
R^\omega f=\lim\limits_{t\to 0^+}\dfrac{W_{\xi,\zeta}(t)f-f}{t}.
$$
Let $\calD_\omega$ be the closure of $\text{dom}(Q^\omega)$ in the $N_\omega$-norm. For each $\omega\in\mathbf{R}$, we define the operator $\R^\omega$ as in \eqref{Aomega-2}, that is
$$
\text{dom}(\R^\omega)=\{f\in\calD_\omega:x\in\text{dom}(R^\omega),R^\omega f\in\calD_\omega\},\quad\R^\omega f=R^\omega f,\, f\in\text{dom}(\R^\omega).
$$
Let $f\in\text{dom}(R)$. Then there exists $\omega\in\mathbf{R}$ such that $f\in\text{dom}(R^\omega)$. We have
$$
R f=R^\omega f=\lim_{t\to 0^+}\dfrac{W_{\xi,\zeta}(t)f-f}{t}\in\Sigma_\omega,
$$
which gives
$$
R f(z)=\lim_{t\to 0^+}\dfrac{W_{\xi,\zeta}(t)f(z)-f(z)}{t},\quad\forall z\in\mathbf{C}.
$$
Consequently, taking into account of forms of $(\zeta_t)_{t\geq 0}$ and $(\xi_t)_{t\geq 0}$ in \eqref{form1}, we get
$$
R f(z)=\dfrac{d}{dt}W_{\xi,\zeta}(t)f(z)\bigg\arrowvert_{t=0}=[H-\ell (aG+b)z]f(z)+\ell (z-G)f'(z),\quad\forall z\in\mathbf{C},
$$
which implies $R\preceq X_2$, and hence
$$
\R^\omega\preceq R^\omega\preceq X_{2,\omega},\quad\forall\omega\in\mathbf{R}.
$$
As mentioned in Proposition \ref{Thm220}, the operator $\R^\omega$ is the generator of the $C_0$-semigroup $(W_{\xi,\zeta}(t)|_{\calD_\omega})_{t\geq 0}$ of bounded, linear operators acting on the Banach space $(\calD_\omega,N_\omega)$. By \cite[Proposition I.5.5, Generation Theorem II.3.8]{EN}, there are constants $\alpha\geq 1$, $\theta\in\mathbf{R}$ such that
$$
N_\omega (W_{\xi,\zeta}(t)f)\leq \alpha e^{\theta t}N_\omega(f),\quad\forall t\geq 0,\forall f\in\calD_\omega,
$$
and moreover $(\theta,\infty)\subseteq\rho(\R^\omega)$.

Let $\lambda\in (\theta,\infty)$. Then the operator $\R^\omega-\lambda I$ is onto. Let $f\in\ker(X_{2,\omega}-\lambda I)$. Then $f\in\text{dom}(X_{2,\omega})\subseteq\calF^2$, and
$$
[H-\lambda-\ell (aG+b)z]f(z)+\ell (z-G)f'(z)=(X_{2,\omega}-\lambda I)f(z)=0,\quad\forall z\in\mathbf{C}.
$$
Letting $z=\zeta_t(w)$, we get
$$
[H-\lambda-\ell(aG+b)\zeta_t(w)]f(\zeta_t(w))+\ell (z-G)f'(\zeta_t(w))=0.
$$
Since
$$\dfrac{\partial \zeta_t(w)}{\partial t}=\ell e^{t\ell}(w-G)=\ell (z-G),$$
the above equation is rewritten as follows
$$
[H-\lambda-\ell(aG+b)\zeta_t(w))f(\zeta_t(w)]+\dfrac{\partial \zeta_t(w)}{\partial t}f'(\zeta_t(w))=0,
$$
which is equivalent to
\begin{eqnarray}\label{revise-tue}
[H-\lambda-\ell(aG+b)\zeta_t(w)]v(t)+v'(t)=0,
\end{eqnarray}
where $v(t)=f(\zeta_t(w))$. Setting
$$u(t)=v(t)e^{-(aG+b)G(1+\ell t-e^{\ell t})-t(\lambda-H)-(e^{\ell t}-1)(aG+b)w},$$
by the product rule for derivatives, we have
\begin{eqnarray*}
v'(t)
&=& u'(t)e^{(aG+b)G(1+\ell t-e^{\ell t})+t(\lambda-H)+(e^{\ell t}-1)(aG+b)w}\\
&& +[-H+\lambda+\ell(aG+b)\zeta_t(w)]v(t),
\end{eqnarray*}
which implies, by \eqref{revise-tue}, that $u'(t)=0$. Hence,
$$
u(t)=u(0),\quad\forall t\geq 0.
$$
Taking into account of the form of $u(t)$, we have
$$
f(\zeta_t(w))=e^{(aG+b)G(1+\ell t-e^{\ell t})+t(\lambda-H)+(e^{\ell t}-1)(aG+b)w}f(w),\quad \forall w\in\C,
$$
and so $W_{\xi,\zeta}(t)f(w)=e^{\lambda t}f(w)$. Thus, we have
$$
\alpha e^{\theta t}N_\omega(f)\geq N_\omega (W_{\xi,\zeta}(t)f)=e^{\lambda t}N_\omega(f),\quad\forall t\geq 0.
$$
Since $\lambda >\theta$, we must have $N_\omega(f)=0$, which means that the operator $X_{2,\omega}-\lambda I$ is one-to-one. By \cite[Lemma 1.3]{KS}, we must have $X_{2,\omega}-\lambda I=\R^\omega-\lambda I$, and hence, $\R^\omega=R^\omega= X_{2,\omega}$.

(2) This conclusion follows directly from Theorem \ref{stone-1}.

(3) Suppose that $\ell>0$ and $aG+b\ne 0$. 

(3-a) We prove by a contradiction that the operator $R$ is not $\calC_{a,b,c}$-selfadjoint. Indeed, assume that the operator $R$ is $\calC_{a,b,c}$-selfadjoint, and hence by \cite{hai2018complex}, it must be maximal, that is
$$
\text{dom}(R)=\{f\in\calF^2:[H-\ell (aG+b)z]f+\ell (z-G)f'\in\calF^2\}.
$$
It is clear that $\mathbf{1}\in\{f\in\calF^2:[H-\ell (aG+b)z]f+\ell (z-G)f'\in\calF^2\}$. Then there exists $\omega\in\mathbf{R}$ such that
$$
\mathbf{1}\in\{f\in\Sigma_\omega:[H-\ell (aG+b)z]f+\ell (z-G)f'\in\Sigma_\omega\}\subseteq\Sigma_\omega.
$$
Since $W_{\xi,\zeta}(t)\mathbf{1}(z)=C(t)K_{\overline{D(t)}}(z)$, we have
$$
\|W_{\xi,\zeta}(t)\mathbf{1}\|=|C(t)|e^{|D(t)|^2/2},
$$
and so,
$$
\sup_{t\geq 0}e^{-\omega t}\|W_{\xi,\zeta}(t)\mathbf{1}\|=\sup_{t\geq 0}e^{t(\re H-\omega)+(e^{\ell t}-\ell t-1)\re[G(aG+b)]+|aG+b|^2\cdot |e^{\ell t}-1|^2/2}=\infty.
$$
Thus, $\mathbf{1}\notin\Sigma_\omega$, but this is impossible.

(3-b) Assume that the operator $R^*$ is the generator of the adjoint semigroup $(W_{\xi,\zeta}(t)^*)_{t\geq 0}$. By Theorem \ref{stone-1}(2), the operator $R$ is $\calC_{a,b,c}$-selfadjoint; but this is impossible by (3-a).

(4) Note that $R$ is a differential operator with a non-maximal domain. Thus, this conclusion follows from \cite{hai2018complex}.
\end{proof}

\begin{rem}
The inverse inclusion in Proposition \ref{generator-2}(4) depends heavily on the structure of the functions
$$
f_m(z)=(z-G)^m e^{z(aG+b)},\quad z\in\C,\,m\in\mathbf{N}.
$$
To see this, we note the reader that
$$
[H-\lambda-\ell (aG+b)z]f_m(z)+\ell (z-G)f_m'(z)=[H-\ell (aG+b)G+m\ell]f_m(z),\quad z\in\mathbf{C}.
$$
Thus, if there exists some $\omega\in\mathbf{R}$ such that $f_m(\cdot)\in\Sigma_\omega$ and $Rf_m(\cdot)\in\Sigma_\omega$, then $H-\ell (aG+b)G+m\ell\in\sigma_p(R)$.
\end{rem}

\section*{Acknowledgements}
The authors would like to thank the referee for insightful comments on the paper.

\bibliographystyle{plain}
\bibliography{refs}
\end{document}